\documentclass[11pt]{article}
\usepackage{amsmath,amsxtra,amssymb,latexsym,
	epsfig,amscd,amsthm,fancybox,epsfig}
\usepackage[mathscr]{eucal}
\usepackage{bbm}
\usepackage{float}
\usepackage{graphicx}
\usepackage{epsfig} 
\usepackage{epstopdf}
\usepackage{cases}
\usepackage{subfigure}
\usepackage{color}
\usepackage{hyperref}

\def\RR{\mathbb R}

\def\NN{\mathbb N}

\def\PP{\mathbb P}

\def\C{\mathcal C}

\def\F{\mathcal F}
\def\K{\mathcal K}
\def\phi{\varphi}
\newtheorem{thm}{Theorem}[section]

\newtheorem{lem}[thm]{Lemma}
\newtheorem{prop}[thm]{Proposition}

\theoremstyle{definition}
\newtheorem{defn}[thm]{Definition}

\newtheorem{exam}[thm]{Example}

\def\tilde{\widetilde}
\def\eps{\varepsilon}
\def\leq{\leqslant}
\def\geq{\geqslant}
\def\bar{\overline}
\def\hat{\widehat}

\newcommand{\abs}[1]{\left\vert#1\right\vert}
\newcommand{\set}[1]{\left\{#1\right\}}

\numberwithin{equation}{section}
\begin{document}
\date{}
\title{Analysis of Reaction-Diffusion Predator-Prey System under Random Switching}

\author{
Nguyen Huu Du\thanks{Department of Mathematics-Mechanics-Informatics,
	Vietnam National University of Science,
	34 Nguyen Trai,
	Thanh Xuan, Ha Noi,	Vietnam, dunh@vnu.edu.vn. This author’s research was supported in part by 
	by Vietnam National Foundation for Science and Technology Development (NAFOSTED) under grant number 101.03-2021.29.} 
\and Nhu Ngoc Nguyen\thanks{
	Department of Mathematics and Applied Mathematical Sciences,	University of Rhode Island,
	5 Lippitt Road,
	Kingston, RI 02881-2018,
	United States, nhu.nguyen@uri.edu. 
	This author’s research was supported in part by the National Science
	Foundation under grant DMS-2407669. 
	This work was finished when the author was visiting Vietnam Institute for Advance Study
	in Mathematics (VIASM). He is grateful for the
	support and hospitality of VIASM.}
}
\maketitle

\begin{abstract}
This paper investigates the long-term dynamics of a reaction-diffusion predator-prey system subject to random environmental fluctuations modeled by Markovian switching. The model is formulated as a hybrid system of partial differential equations (PDEs), where the switching between different ecological regimes captures the randomness in environmental conditions. We derive a critical threshold parameter that determines whether the predator species will eventually go extinct or persist. We further characterize the system’s asymptotic behavior by providing a detailed pathwise description of the $\omega$-limit set of solutions. This analysis reveals how the effects of random switching shape the distribution and long-term coexistence of the species. Numerical simulations are provided to validate and illustrate the theoretical findings, highlighting transitions between different dynamical regimes. To the best of our knowledge, this is the first work that rigorously analyzes a spatially diffusive predator-prey model under Markovian switching, thereby bridging the gap between spatial ecology and stochastic hybrid PDE systems.
\end{abstract}
\medskip



\section{Introduction}\label{s1}

One of the central questions in ecology is: ``Under what conditions will species persist, and when will they face extinction?" This fundamental inquiry lies at the heart of understanding biodiversity, population stability, and ecosystem resilience. To address this, researchers develop mathematical models that simulate the interactions among species, their responses to environmental changes, and their spatial distributions. Such models not only aim to capture the underlying mechanisms governing species survival but also seek to predict scenarios where species might thrive or be at risk. Analyzing these models provides insights into the conditions that support persistence and those that lead to extinction, ultimately helping ecologists and conservationists design strategies to protect vulnerable species and maintain ecosystem balance.
Predator-prey ecological systems are fundamental models in ecology that describe the interactions between species, where one (the predator) feeds on the other (the prey). These systems help illustrate how populations of both species fluctuate over time, often in cycles. Predators depend on prey for survival, while prey populations are controlled by the presence of predators. Such models are crucial for understanding ecosystem dynamics, species persistence, and the balance of biodiversity in natural environments.

{\bf ODE models.} 
The classical (deterministic) Lotka-Volterra equations are commonly used to study these relationships and take the form
\begin{equation}\label{eq-ode}
	\begin{cases}
		\dfrac{du(t)}{dt}=u(t)(a- bu(t)-cv(t)),\\[1ex]
		\dfrac{dv(t)}{dt}=v(t)(-d+ eu(t)-fv(t)),
	\end{cases}
\end{equation}
where $u(t)$, $v(t)$ are the densities of the prey and the predator at time $t \geq 0$ and $a,b,c,d,e,f$ are positive constants.
Such a model was first introduced in \cite{lotka1925elements}. Then, \eqref{eq-ode} and its variants have been used and studied in existing literature by many research in both the mathematics community (e.g., \cite{arditi1989coupling,li2011dynamics,liu2018dynamics,liu2018persistence,zhang2012spatial}) and ecology community (e.g., \cite{beddington1975mutual,dl1975model,holling1959components}).
The parameters in \eqref{eq-ode} interpret biological rates and relationships between predator and prey. For example, 
$a$ stands for the per capita growth rate; $c$ is the predation rate;
$d$ is the death rate of the predator when there is no prey; etc.
For more detail, please see the aforementioned references.

{\bf Stochastic ODE models.}
From a different perspective, the above (deterministic) ODE models in general are not able to describe the effects of various random factors. The inherent randomness of the environment renders population dynamics stochastic. Therefore, it is essential to consider the combined impact of spatial interactions and environmental fluctuations. 
To take into account for these fluctuations and their influence on species persistence or extinction, one approach is to study stochastic predator-prey systems.
Many researchers consider \eqref{eq-ode} with white noise under framework of stochastic differential equations (SDEs); see e.g., \cite{dieu2016protection,du2016conditions, WW2023} and references therein. 

{\bf PDE models.}
However, the model \eqref{eq-ode} does not specify the distribution of the species in areas they are living, its dependence on space, and species' movements.
Therefore, much effort has been devoted to study state-dependence predator-prey model that are able to capture distribution and diffusion in space of species. 
More precisely, let $E$ be a bounded domain in the Euclidean
space $\RR^n$ with sufficiently smooth boundary $\partial E$. 
Assume that species are living in $E$ and let
$u(t,x), v(t,x)$ be real-valued functions of $(t, x)\in [0, \infty)\times E$, which are to be interpreted as population densities of prey and predator at time $t$ and location $x$. 
The classical deterministic spatial predator-prey (see e.g., \cite{leung1978limiting,williams1978nonlinear})  suggests that the densities $u(t,x),v(t,x)$ satisfy the partial differential equation (PDE)
\begin{equation}\label{eq-pde}
	\begin{cases}
		\dfrac{\partial u(t,x)}{\partial t}=\alpha_1\Delta u(t,x) + u(t,x)(a- bu(t,x)-cv(t,x)),\\[1.5ex]
		\dfrac{\partial v(t,x)}{\partial t}=\alpha_2\Delta v(t,x) + v(t,x)(-d+ eu(t,x)-fv(t,x)),\\
u(0,x)=u_0(x)\geq 0,\; v(0,x)=v_0(x)\geq 0,
	\end{cases}
\end{equation}
with the Neumann boundary condition 
\begin{equation}\label{e3}
	\frac{\partial u}{\partial n}= \frac{\partial v}{\partial n}=0 \text{ on } [0,\infty)\times \partial E, 
\end{equation}
where $\alpha_1,\alpha_2>0$ are diffusive constants.
These spatially heterogeneous predator-prey systems and their variants have been widely studied; see e.g., \cite{ai2017traveling,kao2014global,lam2016emergence,li2018asymptotic,Zhao2024,xue2024exponential}.

In general, the model \eqref{eq-pde} is a type of the so-called
 ``reaction-diffusion systems" in PDE community. Such models have been studied extensively  and demonstrate important applications in mathematical biology.
 For example, the works in \cite{bentout2024asymptotic,djilali2024generalities,djilali2024dynamics,djilali2025dynamics1,djilali2025dynamics} and the references therein make a significant progress to bridge PDE models and epidemic models; or the works \cite{li2025lotka,pan2024global} investigate ecological competitive models under the PDE setting; and among others.

{\bf PDE models with randomness.}
While PDE models account for spatial heterogeneity, they remain deterministic. More recently, the second author and a collaborator initiated the study of spatially heterogeneous systems of the form \eqref{eq-pde}, incorporating space-time white noise within the framework of stochastic partial differential equations (SPDEs) \cite{nguyen2019stochastic}; see also \cite{agresti2024reaction,bendahmane2023stochastic,hu2022analysis, qi2025dynamics, xiong2024persistence} and the references therein.

However, many sources of environmental variability in real ecological systems arise from discrete and abrupt random events (or color noise) rather than continuous fluctuations (or white noise). Examples include seasonal cycles that follow predictable annual patterns and extreme weather events such as hurricanes, droughts, and floods. These events can cause sudden shifts in population sizes, disrupt ecological interactions, and even reshape physical landscapes, profoundly affecting species survival and ecosystem stability. Because such phenomena are inherently discontinuous and often follow discrete probabilistic rules, models relying solely on continuous white noise may fail to capture these essential dynamics.

To better reflect this nature of environmental randomness, it is therefore natural to consider stochastic models driven by discrete noise sources, such as Markov chains or telegraph noise. These alternatives offer a more realistic framework for studying ecological systems subject to sudden or periodic environmental disturbances.
A suitable mathematical framework for this model is that of random PDEs with Markov switching, an area that has been receiving increasing attention in recent years. For example, for certain classes of such random switching PDEs, \cite{li2024well, li2023diffusion} establish well-posedness results; \cite{xue2024exponential} investigates stability; \cite{cheng2024space} addresses control problems; \cite{li2025synchronization,zhao2024almost} consider synchronization problems and their application in related fields; and \cite{zhang2024parameters} explores parameter inference; and among others.
However, research on the long-term behavior of random switching PDEs remains relatively scarce, particularly regarding their connections to biological and ecological applications.

Inspired by the above motivation, in this work, we focus on predator-prey model using random PDE with Markovian switching.
Formally, 
let $\xi(t)$ be a Markov process that switches with finite state space.
We now assume that $a, b, c, d, e, f$ in \eqref{eq-pde} are functions of $\xi(t)$ and consider the hybrid reaction-diffusion system under random switching
\begin{equation}\label{eq-main}
	\begin{cases}
		\dfrac{\partial u}{\partial t}=\alpha_1(\xi(t))\Delta u + u(a(\xi(t))- b(\xi(t))u-c(\xi(t))v),\\[1ex]
		\dfrac{\partial v}{\partial t}=\alpha_2(\xi(t))\Delta v + v(-d(\xi(t))+ e(\xi(t))u-f(\xi(t))v),\\[1ex]
		u(0,x)=u_0(x)\geq 0,\;v(0,x)=v_0(x)\geq 0,\;\xi(0)=\xi_0, 
	\end{cases}
\end{equation}
with the Neumann boundary condition \eqref{e3}.

To better illustrate the idea, consider the case that $\xi(t)$ switches (randomly) between two states $\{+,-\}$ in this paper.
Such a stochastic environment switches randomly system \eqref{eq-pde} between two deterministic systems
\begin{equation}\label{e6+}
	\begin{cases}
		\dfrac{\partial u_+(t,x)}{\partial t}=\alpha_1(+)\Delta u_+(t,x) + u_+(t,x)(a(+)- b(+)u_+(t,x)-c(+)v_+(t,x)),\\[1ex]
		\dfrac{\partial v_+(t,x)}{\partial t}=\alpha_2(+)\Delta v_+(t,x) + v_+(t,x)(-d(+)+ e(+)u_+(t,x)-f(+)v_+(t,x)),
		\end{cases}
\end{equation}
and
\begin{equation}\label{e6-}
	\begin{cases}
		\dfrac{\partial u_-(t,x)}{\partial t}=\alpha_1(-)\Delta u_-(t,x) + u_-(t,x)(a(-)- b(-)u_-(t,x)-c(-)v_-(t,x)),\\[1ex]
		\dfrac{\partial v_-(t,x)}{\partial t}=\alpha_2(-)\Delta v_-(t,x) + v_-(t,x)(-d(-)+ e(-)u_-(t,x)-f(-)v_-(t,x)),
\end{cases}
\end{equation}
both are with the Neumann boundary condition \eqref{e3}.



In this study, we begin by rigorously formulating a prey-predator model with state dependence under random switching and establishing its well-posedness. Following this foundational work, our focus shifts to understanding the asymptotic behavior of the population densities 
$u(t,x)$ for prey and 
$v(t,x)$ for predators as time 
$t\to\infty$. Specifically, we aim to classify completely conditions under which a species will either persist or face extinction based on system parameters. Additionally, we investigate the pathwise behavior of the system by thoroughly describing the $\omega$-limit set, which characterizes the long-term dynamics and stability of the system trajectories.
Finally, we present some numerical examples illustrating our theoretical results.

{\bf Our contribution.} We highlight below the main contributions and novelties of our work compared with existing studies:
\begin{itemize}
	\item This study presents the first investigation of a stochastic, state-dependent predator-prey model influenced by random environmental switching.
	A complete characterization of extinction and persistence is provided and the $\omega$-limit set is described.
	These results offer insights both phenomenologically for stochastic reaction-diffusion predator-prey models
	and mathematically for PDEs with random switching.
	\item We introduce novel analytical methods beyond those commonly used in the literature. While the Lyapunov functional approach is standard for studying the long-term behavior of stochastic systems, our analysis focuses on the detailed pathwise behavior of solutions and the concept of Lyapunov exponent from dynamical system theory is used to define a threshold, which characterizes longtime properties.
	 This is particularly challenging due to the combined randomness in both the system’s state and temporal evolution, necessitating new techniques.
\end{itemize}

The rest of paper is organized as follows.
Section \ref{s:2} is devoted to the formulation of the problems, some preliminaries, and definitions.
Section \ref{sec:class} investigates the longtime behavior of the underlying system.
Section \ref{sec:kpp} studies carefully random Fisher-KPP equation, which is used to investigate \eqref{eq-main} but also is of independent interest.
Section \ref{subsec-class} introduces a threshold and proves that its sign characterizes extinction and persistence while Section \ref{sec-omega} describes the $\omega$-limit set of \eqref{eq-main}.
Finally, Section \ref{sec:num} provides some numerical simulations to illustrate theoretical results and Section \ref{sec:con} concludes the paper.

\section{Formulation}\label{s:2}

Let $(\Omega,\{\F_t\}_{t\geq 0},\PP)$ be a probability space satisfying the general hypotheses and $(\xi(t))_{t\geq 0}$ be a Markov
process, defined on $(\Omega,\F,\PP)$, taking values in the set of two elements, say $S= \{+,-\}.$ 
Suppose that
$(\xi(t))_{t\geq 0}$ has the transition intensities $+\xrightarrow{q(+)}-$ and $-\xrightarrow{q(-)}+
$ with $q(+) > 0$, $q(-) > 0$. 
The process $(\xi(t))_{t\geq 0}$ has a unique
stationary distribution
\begin{equation}\label {E2.1}
	\begin{cases}
\pi(+)=\lim_{t\to\infty} \PP\{\xi(t) = +\} = \frac{q(-)}
{q(+) +q(-)},\\
\pi(-)=\lim_{t\to\infty} \PP\{\xi(t) = -\} =\frac{q(+)}{q(+)+ q(-)}.
\end{cases}
\end{equation}
The trajectories of $(\xi(t))_{t\geq 0}$ are piece-wise constant, cadlag functions.


Throughout this paper, $E$ is an open bounded subset of $\mathbb R^n$ and its boundary satisfies some regular condition as described later in Theorem \ref{thm-pre}. 
We denote by $|E|$ its volume and $\bar E$ for its closure.
Denoted by $C(\bar E)$ the space of real-valued continuous functions on $\bar E$ endowed with the sup-norm.
For a function $z\in C(\bar E)$, $\delta\in\mathbb R$, we write $z\geq \delta$ if $z(x)\geq \delta,\;\forall x\in \bar E$.
For simplicity, we will use $[m_1,m_2]$ to denote the interval $[m_1,m_2]$ if $m_1\leq m_2$ and the interval $[m_2,m_1]$ if $m_2<m_1$.
Let $a_{\max}=\max_{i\in\{+,-\}}a(i)$, $a_{\min}=\min_{i\in\{+,-\}}a(i)$, and define similar notions for $b, c, d, e ,f$.

For $\rho_0 > 0$ and $0 < l < 1$,
be fixed numbers, 
$H^{2+l}(\bar E)$ denotes the Banach space
of all real-valued functions $z$ continuous on $\bar E$ with all first and second order
derivatives also continuous on $\bar E$, and with finite value for the norm
\begin{equation*} 
|z|_E^{(2+l)} := \sum_{0\leq |\beta|\leq 2}\sup_E |D^\beta z| + \sum_{|\beta|=2}\sup \frac{|[D^\beta z](x)-[D^\beta z](y)|}{|x-y|^l}
\end{equation*}
where $\beta= (\beta_1,\dots, \beta_n)$ is a multi-index, $|\beta| =\beta_1 + \dots+ \beta_n$, $D^\beta z=
\frac{\partial^{|\beta|}}{\partial x_1^{\beta_1}\dots \partial x_n^{\beta_n}}$. 
The second supremum is taken over all $x$, $y$ in $\bar E$ such that
$0 < |x -y| < \rho_0$.

We start with well-posedness and basis properties of the solution to \eqref{eq-main}.
Here, we consider solutions in the classical sense. Particularly, a $\F_t$-adapted stochastic process $(u(t,x),v(t,x))$ is a solution to \eqref{eq-main} if almost surely, $u(t,x),v(t,x)$ are differentiable in $t$, twice differentiable in $x$ and satisfy \eqref{eq-main}.

\begin{thm}\label{thm-pre}
	Let $E$ be an open bounded connected subset of
	$\RR^n$, $n\geq1$. 
	Assume that $\partial E \in H^{2+l}$ for some $l$ satisfying $0 < l < 1$. 
	Let $u_0$, $v_0$ be non-negative valued functions in $H^{2+l}(\bar E)$ with vanishing normal derivatives on $\partial E$. 
	Then, we have the following claims.
	\begin{itemize} 
\item[{\rm (i)}] Equation \eqref{eq-main} with initial condition $(u_0,v_0)$ and boundary conditions \eqref{e3} has a
	unique solution $(u(t,x)$, $v(t,x))$.
	Moreover, $u(t,\cdot)\in H^{2+l}(\bar E)$, $v(t,\cdot)\in H^{2+l}(\bar E)$ for all $t>0$ a.s.
\item[{\rm(ii)}] $u(t,x)\geq 0$, $v(t,x)\geq 0$ for all $(x, t)\in \bar E\times [0,\infty)$ a.s.
\item[{\rm (iii)}] If $u_0$, $v_0$ are not identically zero, then both $u(t,\cdot)$, $v(t,\cdot)$ are strictly positive for all $t >0$ a.s.
\item[{\rm (iv)}]
The solution is bounded almost surely. More precisely, with probability 1
\begin{equation}
	0\leq u(t,x)\leq M_1,\;0\leq v(t,x)\leq M_2,\;\forall (t,x)\in (0,\infty)\times \bar E,
\end{equation}
where
$$M_1>\max\Big\{\max_{i\in\{+,-\}}\frac{a(i)}{b(i)},\sup_{x\in\bar E}u_0(x)\Big\},$$
and
$$ M_2>\max\Big\{\max_{i\in\{+,-\}}\frac{a(i)}{c(i)},\max_{i\in\{+,-\}}\frac{M_1e(i)-d(i)}{f(i)},\sup_{x\in\bar E}v_0(x)\Big\}.$$

	\end{itemize}
\end{thm}

\begin{proof}
The proof of this theorem for deterministic systems \eqref{e6+} and \eqref{e6-} can be found in \cite{leung1978limiting}.
	It is noted that \eqref{eq-main} is switching system that switch (randomly based on the changes of $\xi(t)$) between two deterministic systems \eqref{e6+} and \eqref{e6-}, and $M_1$, $M_2$ are independent of $\xi(t)$. Therefore, the proof this theorem can follow \cite{leung1978limiting}.
\end{proof}

One of our main motivations comes from ecology, where ecologists seek to determine whether a particular ecological population is at risk of extinction or can persist over time since understanding these dynamics is crucial for conservation efforts and managing biodiversity. To address this, we introduce the following concepts.
\begin{defn}
	We say a specie with density $z(t,x)$ is extinct in strong sense if 
	$$\lim_{t\to\infty} \sup_Ez(t,x)=0 \text{ a.s.}$$
\end{defn}

\begin{defn}
	We say a specie with density $z(t,x)$ is persistent in average sense if 
	\begin{equation*}
		\liminf_{T\to\infty}\frac 1T\int_0^T \frac 1{|E|}\int_E z(t,x)dxdt>\delta>0,
	\end{equation*}
	for some $\delta$, which is independent of initial condition.
\end{defn}

We also want to study further the limit set behavior of systems.
Adapting the definition introduced in \cite{du2011dynamics}, we define $\omega$-limit set as follows.
\begin{defn}
	The $\omega$-limit set of the trajectory starting from an initial value $(u_0, v_0)$ is
	$$
	\Omega(u_0,v_0,\omega) =\displaystyle\bigcap\limits_{T>0}\bar{
		\bigcup\limits_{t>T}
		(u(t,\cdot,\omega), v(t,\cdot,\omega))},
	$$	
	with the closure being obtained under supremum norm in $C(\bar E)$,
	where $(u(t,\cdot,\omega), v(t,\cdot,\omega))$ is solution to \eqref{eq-main} with initial condition $(u_0,v_0)$.
	
\end{defn}

This definition generalizes the concept in \cite{du2011dynamics} to functional space. As remarked in \cite{du2011dynamics}, this type of $\omega$-limit set is different from the one in \cite{crauel1994attractors} but it is closest to that of an $\omega$-limit set for a deterministic dynamical system. In the case where $\Omega(u_0, v_0,\omega)$ is a.s constant, it is similar to the concept
of weak attractor and attractor given in \cite{mao2001attraction}. Although, in general, the $\omega$-limit set in this sense
does not have the invariant property, this concept is appropriate for our purpose of describing the
pathwise asymptotic behavior of the solution with a given initial value.

\section{Asymptotic Behavior}\label{sec:class}
This section investigates the long-term behavior of the system \eqref{eq-main}, with a particular focus on classifying conditions for extinction and persistence. To this end, we introduce a threshold parameter 
$\lambda$, whose sign is shown to determine these asymptotic properties. In Section \ref{sec:kpp}, we first study a random Fisher–KPP equation in detail. This equation not only plays a key role in the definition of 
$\lambda$ but is also of independent interest. Section \ref{subsec-class} then formally defines the threshold $\lambda$ and establishes that its sign characterizes extinction and persistence. Finally, Section \ref{sec-omega} describes the 
$\omega$-limit set of solutions to \eqref{eq-main}.
\subsection{Random Fisher-KPP equation}\label{sec:kpp}
To analyze the long-term behavior of the reaction-diffusion predator-prey system with random switching \eqref{eq-main}, it is essential to first examine the related boundary problem (i.e., the related system when $v(t,x)=0$). Specifically, this involves studying the random  Fisher-KPP equation. Consequently, this section is dedicated to investigating the random Fisher-KPP equation.

The deterministic Fisher-KPP equation (see e.g., \cite{fisher1937wave,kolmogorov1937etude}) is given by
\begin{equation}\label{e0}
\begin{cases}
\dfrac{\partial\tilde{u}(t,x)}{\partial t}=\alpha_1\Delta \tilde{u}(t,x) + \tilde{u}(t,x)(a- b\tilde{u}(t,x)),\\[1ex]
\tilde{u}(0,x)=u_0(x)\geq 0, x\in E,
\end{cases}
\end{equation}
with the boundary condition 
\begin{equation}\label{bc}\frac{\partial \tilde{u}}{\partial n}=0 \text{ on } [0,\infty)\times \partial E.
\end{equation}
 It is well-known that if $u_0\geq 0$ then \eqref{e0} has a unique solution $\tilde{u}(t,x)$ bounded by $\max\set{\frac ab, \sup_{\bar E} u_0(x)}$ on  $(0,\infty)\times \bar E$ and $\lim_{t\to\infty} U(t,x)=\frac ab$ uniformly on $\bar E$. 

Now, taking account the effect of random switching into the consideration, consider the switching system 
\begin{equation}\label{e1}
\begin{cases}
\dfrac{\partial \tilde{u}(t,x)}{\partial t}=\alpha_1(\xi(t))\Delta \tilde{u}(t,x) + \tilde{u}(t,x)(a(\xi(t))- b(\xi(t))\tilde{u}(t,x)),\\[1ex]
\tilde{u}(0,x)=u_0(x)\geq 0, x\in E,
\end{cases}
\end{equation}
with the boundary condition \eqref{bc}.
It is easy to see that \eqref{e1} has a unique solution $\tilde{u}(t,x)$ bounded by $\max\set{\max_{i\in\{+,-\}}\frac {a(i)}{b(i)}, \sup_{\bar E} u_0(x)}$ on  $(0,\infty)\times\bar E$.

Next, we investigate the $\omega$-limit set of the solution \eqref{e1}.   
Let
\begin{equation}\label {E2.2} 0 = \tau_0 < \tau_1 < \tau_2<...< \tau_n < ...
\end{equation}
be the jump times of $\xi(t)$. Put
\begin{equation}\label {E2.3} 
\sigma_{1} = \tau_1 - \tau_0, \; \sigma_{2}= \tau_2 - \tau_1,...,\sigma_{n} = \tau_n - \tau_{n-1}...
\end{equation}
$\sigma_{1}= \tau_{1} $ is the first exile from the initial state, $\sigma_{2}$ is the time the process $(\xi(t))_{t\geq 0}$ spends in the state into which it moves from the first state, etc. It is known that $(\sigma_{k})_{k=1}^\infty$ are independent in the condition of given sequence $(\xi(\tau_{k}))_{k=1}^\infty$.
Note that 
if $\xi_0 = +$ then $\sigma_{2n+1}$ has the exponential density $\alpha 1_{[0,\infty)}\exp(-\alpha t)$ and $\sigma_{2n}$ has the density $\beta 1_{[0,\infty)}\exp(-\beta t)$. Conversely, if $\xi_0 = -$ then $\sigma_{2n}$ has the exponential density $\alpha 1_{[0,\infty)}\exp(-\alpha t)$ and $\sigma_{2n+1}$ has the density $\beta 1_{[0,\infty)}\exp(-\beta t)$ (see \cite[vol. 2, pp. 217]{gihman2012controlled}). Here $1_{[0,\infty)}=1$ for $t\geq 0$ ($=0$ for $t<0$).
Moreover, let
$
\mathcal F^n_0 = \sigma(\tau_k: k\leq n); \mathcal F^\infty_n = \sigma(\tau_k -\tau_n: k > n),
$
then $(u(\tau_n,\cdot), v(\tau_n,\cdot))$ is $\mathcal F^n
_0$-measurable and if $\xi(0)$ is given then $\mathcal F^n
_0$ is independent of $\mathcal F^\infty_n$ .


It is seen that if  $y(t)$ is the solution of (spatial homogeneous) logistic equation with random switching
 \begin{equation}\label{e11}
\begin{cases}
\dfrac{dy(t)}{dt}= y(t)(a(\xi(t))- b(\xi(t))y(t)),\\
y(0)=y_0\geq 0,\\
\end{cases}
\end{equation}
then the function $\tilde{u}(t,x)=y(t), x\in \bar E$ is the solution of \eqref{e1} with the initial condition $u_0(x)\equiv y_0$. It is also known that $(y(t),\xi(t))$ is a Markov process and it has  a unique ergodic stationary distribution in $(0,\infty)\times\{+,-\}$, namely $\pi(dy,di)$. Further, by solving Fokker-Planck equation  we have
$\pi=\delta_{\{+\}}\mu^+(y)+\delta_{\{-\}}\mu^-(y),$
where
$$\begin{aligned}
\mu^+(y)=\theta\frac{\abs{a(+)-b(+)y}^{\frac{q(+)}{a(+)}-1}\abs{a(-)-b(-)y}^{\frac{q(-)}{a(-)}}}{y^{\frac{q(+)}{a(+)}+\frac{q(-)}{a(-)}+1}},\\
\mu^-(y)=\theta\frac{\abs{a(+)-b(+)y}^\frac{q(+)}{a(+)}\abs{a(-)-b(-)y}^{\frac{q(-)}{a(-)}-1}}{y^{\frac{q(+)}{a(+)}+\frac{q(-)}{a(-)}+1}},\end{aligned}$$
 and 
$$
\theta=\bigg\{\int_{\frac{a(+)}{b(+)}}^{ \frac{a(-)}{b(-)}}  \Big(\pi(+)\mu^+(y)+ \pi(-)\mu^-(y)\Big)dy\bigg\}^{-1}.
$$

\begin{lem}
For any $y_0>0$, the $\omega$-limit set $\Omega_y(y_0)$ of $y(t),t\geq 0$ is $\left[\frac{a(+)}{b(+)},\frac{a(-)}{b(-)}\right]$.
Where the $\omega$-limit set $\Omega_y(y_0)$ of ODE \eqref{e11} is defined by
$
\Omega_y(y_0,\omega) =\displaystyle\bigcap\limits_{T>0}\bar{
	\bigcup\limits_{t>T}
	y(t,\omega)}.
$
\end{lem}
\begin{proof}
	It is known that if $\xi(t)=+$ (resp. $-$) for all $t\geq 0$, then the corresponding deterministic of \eqref{e11} converges to the equilibrium $\frac{a(+)}{b(+)}$ (resp. $\frac{a(-)}{b(-)}$). On the other hand,
for any $k\in\mathbb N$ and $0<s_1 <s_2 $, $0<s_3 <s_4$,
if we	let 
	$A_k:=\{\omega: s_1<\sigma_{k}<s_2, s_3<\sigma_{k+1}<s_4\},$
then using arguments in the proof of \cite[Theorem 2.2]{du2011dynamics}, we obtain that
	${\mathbb P}\biggl(\bigcap_{k=1}^{\infty}\bigcup_{i=k}^{\infty}A_i\biggl)=1.$
	As a result,
	\begin{equation*}
		{\PP}\{\omega: s_1<\sigma_{n}<s_2, s_3<\sigma_{n+1}<s_4\mbox{ i.o. of}\;n\}=1.
	\end{equation*}
	Therefore, we can prove that $\Omega_y(y_0,\omega)=[\frac{a(+)}{b(+)},\frac{a(-)}{b(-)}]$.
\end{proof}
 Denote 
by $\K^*$ the subset  of $C(\bar E)$ consisting of the constant functions valued in $\left[\frac{a(+)}{b(+)},\frac{a(-)}{b(-)}\right]$.
\begin{prop}\label{prop2.1-1}
For any $u_0\geq0, u_0\not\equiv 0$, we have 
 the  $\Omega_{\tilde{u}}(u_0)=\K^*$, where $\Omega_{\tilde{u}}(u_0)$ is the $\omega$-limit set  of solution $\tilde{u}(t,\cdot)$ to \eqref{e1}, starting from $u_0$.
\end{prop}
\begin{proof}
By maximum principle (see e.g., \cite[Theorem 2]{conway1977comparison}) we see that for all $t> 0$,
$
\inf_{x\in \bar E} \tilde{u}(t,x)>0. $
Therefore, without loosing the generality, we can suppose that $\inf_{x\in \bar E} \tilde{u}(0,x)>0$. 

Let $y_1(t)$ be the solution of \eqref{e11} with $y_1(0)=\min_{x\in E} \tilde{u}(0,x)>0$ and $y_2(t)$ be the solution of \eqref{e11} with $y_2(0)=\max_{x\in E} \tilde{u}(0,x)>0$. Then, 
by comparison principle (see e.g., \cite{conway1977comparison}), we have
$$y_1(t)\leq \tilde{u}(t,x)\leq y_2(t),\;\forall x\in E.$$
Now, we will prove that with probability 1,
\begin{equation}\label{e11b}\lim_{t\to\infty}(y_1(t)-y_2(t))=0. \end{equation}
Indeed, we have  
$\frac d{dt}\left(\frac 1{y_1(t)}-\frac 1{y_2(t)}\right)=-a(\xi_t)\left(\frac 1{y_1(t)}-\frac 1{y_2(t)}\right),$
which implies that
$\lim_{t\to\infty}\left(\frac 1{y_1(t)}-\frac 1{y_2(t)}\right)=0.$
On the other hand, 
$\left|\frac 1{y_1(t)}-\frac 1{y_2(t)}\right|
\geq \frac{\left|{y_1(t)}-{y_2(t)}\right|}{\big(\max\{\max_{i\in\{+,-\}}\frac {a(i)}{b(i)}, y_1(0),y_2(0)\}\big)^2},$
where we used that fact that almost surely
the solution to \eqref{e11} is bounded by
$\max\{\max_{i\in\{+,-\}}\frac {a(i)}{b(i)}, y_0\}$. 
Therefore, 
$\lim_{t\to\infty}(y_1(t)-y_2(t))=0.$
Because of $\Omega_{y_1}=\Omega_{y_2}=\big[\frac{a(+)}{b(+)},\frac{a(-)}{b(-)}\big]$ and \eqref{e11b}, we have the proof.
\end{proof}
\subsection{Extinction and Persistence of Random Reaction-Diffusion Prey-predator Equation}\label{subsec-class}

We want to (fully) classify whenever  $v(t,x)$ (or $u(t,x)$) is extinct and when it is permanent. We will introduce a threshold number and show that its sign will determine 
persistence and extinction.


Recall that $\pi(dy,di)$ is the invariant measure of $(y(t),\xi(t))$, which is the solution to \eqref{e11}.
Denote 
\begin{equation}\label{e8}
	\begin{aligned} 
		\lambda=\int_{(0,\infty)\times\{+,-\}}(-d(i) +e(i)y)\pi(dy,di).
	\end{aligned}
\end{equation}
It is seen that $\lambda$ is well defined due to the boundedness of $y(t)$.
The sign of $\lambda$ is shown to fully classify the longtime behavior of the random reaction-diffusion system \eqref{eq-main}.
In particular, Theorem \ref{thm-ext} shows that if 
$\lambda<0$, the predator population goes extinct, while the prey density converges to the solution of a Fisher–KPP equation. In contrast, when 
$\lambda>0$, both species persist, as established in Theorem \ref{thm-per}.

\begin{thm}\label{thm-ext}
	 If $\lambda <0$ then 
	\eqref{eq-main} is extinct in strong sense with exponential rate, that means, with probability 1,
	\begin{equation}\label{e3.10}
		\lim_{t\to\infty} \dfrac{\ln(\sup_Ev(t,x))}t\leq\lambda<0.
		\end{equation}
Further,  
	\begin{equation}\label{e3.11}\lim_{t\to\infty} \sup_E(\tilde{u}(t,x)-u(t,x))=0,\end{equation}
	where $\tilde{u}$ is the solution to random Fisher-KPP equation \eqref{e1}.
\end{thm} 
\begin{proof} 
	Let $y(t)$  be the solution of the spatial homogeneous equation \eqref{e11}
	with $y(0)=\max_{x\in\Omega}u(0,x).$ Then, with probability 1,
	$$u(t,x)\leq y(t) \text{ for all }x\in E.$$
	Therefore, 
	$$\frac{\partial v(t,x)}{\partial t}\leq  \alpha_2(\xi(t))\Delta v(t,x)+ v(t,x)(-d(\xi(t))+ e(\xi(t))y(t)).$$
	By comparison theorem, one has
	$v(t,x)\leq \bar V(t)\text{ a.s.},$
	where $\bar V(t)$ is the solution of the spatial homogeneous equation 
	$$\frac{d \bar V(t)}{d t}= \bar V(t)(-d(\xi(t))+ e(\xi(t))y(t)),$$
	with the initial condition $\bar V(0)=\sup_{\bar E} v(0,x)$.
By ergodicity of $(V(t),\xi(t))$, 
	$$
	\begin{aligned}
	\lim_{t\to\infty}\frac {\ln \bar V(t)}t&=\lim_{t\to\infty}\frac 1t\int_0^t(-d(\xi(s))+ e(\xi(t))y(s))ds\\
	&=\int_{(0,\infty)\times\{+,-\}}(-d(i) +e(i)y)\pi(dy,di)=\lambda<0\text{ a.s.}.
	\end{aligned}
	$$
	It follows that 
	$$	\lim_{t\to\infty} \dfrac{\ln(\sup_Ev(t,x))}t\leq \lim_{t\to\infty}\frac {\ln \bar V(t)}t=\lambda<0\text{ a.s.}$$
Thus, \eqref{e3.10} is proved.

By  \eqref{e3.10},  for any $\omega$ being outsite of a null set, we can find  $T=T(\omega)>0$ 
such that
$\sup_Ev(t,x))\leq \frac{e^{\frac\lambda 2t}}{c_{\max}}<a_{\min} \;\; \forall\;\; t\geq T.$
Therefore, 
$$
\begin{aligned}
\alpha_1(\xi(t))\Delta u+u(a(\xi(t))-e^{\frac\lambda 2t}-b(\xi(t))u\leq \frac{\partial u}{\partial t}\leq \alpha_1(\xi(t))\Delta u+u(a(\xi(t))-b(\xi(t))u).
\end{aligned}
$$
Let $\hat y(t)$ be the solution of the equation $\frac{d\hat y(t)}{dt}=\hat y(t)(a(\xi(t))-e^{\frac\lambda 2t}-b(\xi(t))\hat y(t))$ for $t\geq T$ with the initial condition $\hat y(T)= \inf_E u(T,x)$ and $y(t)$ be the solution of \eqref{e11} with the initial condition $y(T)= \sup_E u(T,x)$. Then, by comparison theorem, we have
\begin{equation}\label{e3.12}\hat y(t)\leq u(t,x)\leq\tilde{u}(t,x)\leq y(t), \;\forall x\in E,\; \forall  t\geq T.
	\end{equation}
Put $z(t)=\frac 1{\hat y(t)}-\frac 1{y(t)}$. Direct computations show that
$\frac {dz(t)}{dt}=-a(\xi(t))z(t)+\frac{e^{\frac\lambda 2t}}{\hat y(t)},$
which implies that
$z(t)=e^{-A(t)}\left(z(T)+\int_T^te^{A(s)}\frac{e^{\frac\lambda 2s}}{\hat y(s)}\right),$
where $A(t)=\int_T^ta(\xi(s)ds.$ Noting that both $\hat y(t)$ and $y(t)$  are bounded below and above by positive constants and $\lim_{t\to\infty} A(t)=\infty$  we follow that
$\lim_{t\to\infty}z(t)=0.$
Therefore, it is easily seen that
$\lim_{t\to\infty}(y(t)-\hat y(t))=0.$
This together with \eqref{e3.12} completes the proof.
\end{proof}

\begin{thm}\label{thm-per}
	If $\lambda>0$ then for any initial condition $u_0\geq 0$, $v_0\geq0$, not identical to $0$,
	\eqref{eq-main} is permanent in average sense. That means
			\begin{equation}\label{e3.2}
		\liminf_{T\to\infty}\frac 1T\int_0^T \frac 1{|E|}\int_E u(t,x)dxdt>\delta_1>0,\text{ and }
	\end{equation}
	\begin{equation}\label{e3.3}
		\liminf_{T\to\infty}\frac 1T\int_0^T \frac 1{|E|}\int_E v(t,x)dxdt>\delta_2>0,
	\end{equation}
	where $\delta_1,\delta_2$ are positive constants, which are independent of initial conditions $(u_0,v_0)$.
\end{thm}
\begin{proof}
	By maximum principle (see e.g., \cite[Theorem 2]{conway1977comparison}) we see that for all $t> 0$,
	$
	\inf_{x\in \bar E}u(t,x)>0$, $\inf_{x\in \bar E}v(t,x)>0$.
	Therefore, without loosing the generality, we can suppose that $\inf_{x\in \bar E} u_0(x)>0$, $\inf_{x\in \bar E} v_0(x)>0$.

We will prove \eqref{e3.3} first.	We have from direct calculations and divergence theorem that
	\begin{equation}
		\begin{aligned}
			\frac d{dt}\int_E\ln v(t,x)dx
			=&\int_E \frac{\alpha_2(\xi(t))\Delta v(t,x)}{v(t,x)}dx+\int_E (-d(\xi(t))+e(\xi(t))u(t,x)-f(\xi(t))v(t,x))dx\\
			=&\alpha_2(\xi(t))\int_E \frac{\nabla v(t,x)}{v^2(t,x)}dx+|E| (-d(\xi(t))+e(\xi(t))y(t))\\
			&+\int_E e(\xi(t))\big(u(t,x)-y(t)\big)dx-\int_E f(\xi(t))v(t,x)dx.
		\end{aligned}
	\end{equation}
	This combines with the fact that $\limsup_{t\to\infty}\frac 1t\int_E\ln v(t,x)dx\leq 0$ and that
	$$
	\lim_{T\to\infty}	\frac 1T\int_0^T(-d(\xi(t))+e(\xi(t))y(t))dt=\lambda,
	$$
	we have
	\begin{equation}
		\liminf_{T\to\infty}	\frac 1T\int_0^T\Big(\int_E e(\xi(t))\big(y(t)-u(t,x)\big)dx+\int_E f(\xi(t))v(t,x)dx\Big)dt\geq |E|\lambda.
	\end{equation}
	This implies
	\begin{equation}\label{eq-1-1}
		\liminf_{T\to\infty}	\frac 1T\int_0^T\Big(e_{\max}\int_E \big(y(t)-u(t,x)\big)dx+f_{\max}\int_E v(t,x)dx\Big)dt\geq |E|\lambda.
	\end{equation}
	On the other hand, we have
	\begin{equation}\label{eq-317}
		\begin{aligned}
			\frac d{dt}\int_E\ln u(t,x)dx
			=&\int_E \frac{\alpha_1(\xi(t))\Delta u(t,x)}{u(t,x)}dx+\int_E (a(\xi(t))-b(\xi(t))u(t,x)-c(\xi(t))v(t,x))dx\\
			=&\alpha_1(\xi(t))\int_E \frac{\nabla u(t,x)}{u^2(t,x)}dx+|E| (a(\xi(t))-b(\xi(t))y(t))\\
			&+\int_E b(\xi(t))\big(y(t)-u(t,x)\big)dx-\int_E c(\xi(t))v(t,x)dx.
		\end{aligned}
	\end{equation}
	This combines with the fact that $\limsup_{t\to\infty}\frac 1t\int_E\ln u(t,x)dx\leq 0$ and that
	$$
	\lim_{T\to\infty}	\frac 1T\int_0^T(a(\xi(t))-b(\xi(t))y(t))dt=\lim_{t\to\infty}\frac 1t\ln y(t)=0,
	$$
	we have
	\begin{equation}
		\liminf_{T\to\infty}\frac 1T\int_0^T\Big( \int_E c(\xi(t))v(t,x)dx-\int_Eb(\xi(t))\big(y(t)-u(t,x)\big)dx\Big)dt\geq 0.
	\end{equation}
	Therefore, one has
	\begin{equation}\label{eq-1-2}
		\liminf_{T\to\infty}\frac 1T\int_0^T\Big( c_{\max}\int_E v(t,x)dx-b_{\min}\int_E\big(y(t)-u(t,x)\big)dx\Big)dt\geq 0.
	\end{equation}
	Multiplying both sides of \eqref{eq-1-1} and \eqref{eq-1-2} with $b_{min}$ and $e_{max}$, and adding side by side, we have
	\begin{equation}
		\liminf_{T\to\infty}\frac 1T\int_0^T\frac 1{|E|}\int_Ev(t,x)dxdt\geq \frac{\lambda b_{\min}}{f_{\max}b_{\min}+c_{\max}e_{\max}}>0.
	\end{equation}

	Next, we prove \eqref{e3.2}.	We have from divergence theorem that
		\begin{equation}\label{eq-323}
			\begin{aligned}
					0&=\liminf_{T\to\infty}	\frac 1T\Big(\int_E u(T,x)dx-\int_Eu_0(x)dx\Big)\\
					&=\liminf_{T\to\infty}\frac 1T\int_0^T\int_E (a(\xi(t))u(t,x)-b(\xi(t))u^2(t,x)-c(\xi(t))u(t,x)v(t,x))dxdt\\
&\leq \liminf_{T\to\infty}\frac 1T\int_0^T\Big(a_{\max}\int_E u(t,x)dx-c_{\min}\int_E u(t,x)v(t,x)dx\Big)dt					
									\end{aligned}
	\end{equation}
Doing similarly with $\int_Ev(t,x)dx$ we can get that
	\begin{equation}\label{eq-324}
d_{\min}\liminf_{T\to\infty}\frac 1T\int_0^T\int_E v(t,x)dxdt\leq e_{\max}\liminf_{T\to\infty}\frac 1T\int_0^T\int_Eu(t,x)v(t,x)dxdt.
	\end{equation}
Combining \eqref{eq-323} and \eqref{eq-324}, we get
\begin{equation}
	\liminf_{T\to\infty}\frac1T\int_0^T\int_Eu(t,x)dxdt\geq\frac{c_{\min}d_{\min}}{a_{\max} e_{\max}}\liminf_{T\to\infty}\frac1T\int_0^T\int_Ev(t,x)dxdt.
\end{equation}	
This together with \eqref{e3.3} proves \eqref{e3.2}.	
\end{proof}

\subsection{$\omega$-limit Set of Random Reaction-Diffusion Prey-predator Equation}\label{sec-omega}
In this section, we will go further to investigate the $\omega$-limit set of the solution.
We start with the following lemma regarding the behavior of deterministic system \eqref{eq-pde}.
\begin{lem}\label{lem-com}  
	Consider the deterministic system \eqref{eq-pde}.
	Let $K$ is a compact in $C(\bar E)$ such that $z\in K$ implies $z\geq \delta$ for some $\delta>0$.	
	\begin{itemize}
		\item[]{\rm (i)} If $\frac ab>\frac de$ then, \eqref{e1} has a unique  equilibrium point 
		$$(u^*,v^*)=\Big(\frac {af+cd}{bf+ce},\frac {ae-bd}{bf+ec}\Big),$$
		and
		 $\lim_{t\to\infty} (u(t,x),v(t,x))=(u^*,v^*)$ uniformly on $E$ and  on $K$. Further, for any $\eps>0$,  there is a $T=T(\eps)$ such that 
		$
		(u(t,\cdot),v(t,\cdot))\in U_\eps(u^*,v^*)$ for all $t\geq T$, 
		provided $u(0,x),v(0,x)\in  K,$
		where $U_\eps(u^*,v^*)$ is $\eps$-neighbor of constant functions $(u^*,v^*)$ in $C(\bar E)$.
		
		\item[]{\rm (ii)} If $\sup_Ev(t,x)\to 0$ then $u(t,x)\to\frac{a}{b}$ uniformly in $E$.
		
		\item[]{\rm (iii)} 
		If $\frac ab<\frac de$ then $\lim_{t\to\infty} (u(t,x),v(t,x))=\Big(\frac ab,0\Big)$ uniformly on $E$ and this convergence is uniform in initial condition $u_0\in K$, and $v_0\geq 0\in H^{2+l}(\bar E)$.
	\end{itemize} 
\end{lem}
\begin{proof}
We first prove (i). For each $\bar u,\bar v\in K$, 
	by \cite{leung1978limiting}, the solution $ (u(t,x),v(t,x))$ with the initial condition $ (u(0,x),v(0,x))=(\bar u,\bar v)$ satisfies 
	$$
	\lim_{t\to\infty} (u(t,x),v(t,x))\to (u^*,v^*)\text{ uniformly in }E,
	$$
	there is $t_{\bar u,\bar v}>0$ such that $(u(t,\cdot),v(t,\cdot))\in U_{\eps/2}(u^*,v^*)$ for any $t\geq t_{\bar u,\bar v}$.
	Therefore, there is a $\delta_{\bar u,\bar v}>0$ such that
	for any initial data $(u'_0,v'_0)\in U_{\delta_{\bar u,\bar v}}((\bar u,\bar v)$ we have  
	$(u(t,\cdot),v(t,\cdot))\in U_{\eps}(u^*,v^*)$ for all $t\geq t_{\bar u,\bar v}$.
	Let consider the open covering $\{U_{\delta_{\bar u,\bar v}}((\bar u,\bar v): \bar u,\bar v\in K\}$. Since $K$ is compact, there is a finite covering 
	$\{U_{\delta_{\bar u_i,\bar v_i}(\bar u_i,\bar v_i)}: i=1,2,...,N\}$ of $K$. Put $ T=\max_{i=1,\dots,N}t_{\bar u_i,\bar v_i}$ we have the proof.
	
Next, we prove (ii).
Let $\eps>0$ be arbitrary. There is $t_1$ such that $t>t_1$ then $v(t,x)<\frac{\eps}c.$ By comparison argument then $u_\eps(t,x)\leq u(t,x)\leq \tilde{u}(t,x)$ where $\tilde{u}$ is solution to the deterministic Fisher-KPP equation \eqref{e0} and $u_\eps$ is solution 
\begin{equation*}
	\begin{cases}
		\dfrac{\partial u_\eps(t,x)}{\partial t}=\alpha_1\Delta u_\eps(t,x) + u_\eps(t,x)(a-\eps- bu_\eps(t,x)),\;t\geq t_1,\\
		u_\eps(t_1,x)=u(t_1,x).
	\end{cases}
\end{equation*}
Therefore, we can obtain $u(t,x)\to \frac{a}{b}$ uniformly in $E$.

The part (iii) following the fact that $v(t,x)\to 0$ whenever $ae<bd$, which can be proved similarly as done in Theorem \ref{thm-ext}.
\end{proof}
Now, let $M=\max\{M_1,M_2\}$ where $M_1,M_2$ are given in Theorem \ref{thm-pre}. It is noted that due to embedding theorem, for any constant $\delta>0$, the space $\C_{\delta,M}$ defined by
$$\{(z_1,z_2)\in \text{closure}(H^{2+l}(\bar E))\times \text{closure}(H^{2+l}(\bar E)): \delta\leq z_1(x),z_2(x)\leq M,\;\forall x\}$$ is compact in $\C(\bar E)$.
\begin{thm}\label{thm-omega-ext}
	If $\lambda<0$, for any initial condition $(u_0,v_0)$ then $\Omega(u_0,v_0,\omega)=\K^*\times\{0\}$.
\end{thm}
\begin{proof}
This theorem follows directly Theorem \ref{thm-ext}.
\end{proof}

Next, we want to describe the $\omega$-limit set of \eqref{eq-main} when $\lambda>0$. 
In next two lemmas, we will show that $u(t,x)$ and $v(t,x)$ is greater than some small positive constant, for all $x$, infinitely often.

\begin{lem}\label{lem-sup} Assume that $\lambda>0$. There is $\delta>0$ such that
	with probability 1, there are infinitely many $t_n = t_n(\omega) > 0$ such that $t_n > t_{n-1}$, $\lim_{n\to\infty} t_n =\infty$
and	$$\sup_{\bar E}u(t_n,x)\geq \delta\text{ and }\sup_{\bar E}v(t_n,x)\geq \delta,\;\forall n.$$
\end{lem}
\begin{proof} 
	This lemma follows directly from Theorem \ref{thm-per}.
\end{proof}

\begin{lem}\label{lem-inf}
Assume that $\lambda>0$. There is $\delta_1>0$ such that
with probability 1, there are infinitely many $s_n = s_n(\omega) > 0$ such that $s_n > s_{n-1}$, $\lim_{n\to\infty} s_n =\infty$
and	$$\inf_{\bar E}u(s_n,x)\geq \delta_1\text{ and }\inf_{\bar E}v(s_n,x)\geq \delta_1,\;\forall n.$$
\end{lem}
\begin{proof}
	By Lemma \ref{lem-sup}, there are a sequence $\{t_n\}\uparrow\infty$ such that
	$$
	\sup_E u(t_n,x)\geq \delta>0,\;	\sup_E v(t_n,x)\geq \delta>0,
	$$
	for some $\delta$.
Using Harnack inequality \cite[Theorem 3]{aronson1967local}, we have that there are $\{s_n\}\uparrow\infty$ such that
\begin{equation}
\sup_{E}u(t_n,x)\leq C\inf_E u(s_n,x), \;	\sup_{E}v(t_n,x)\leq C\inf_E v(s_n,x)
\end{equation}
where $C<\infty$ is a constant, which is independent of $n$.
It is noted that we can choose the constant $k=0$ in \cite[Theorem 3]{aronson1967local} since
$|u(a-bu-cv)|\leq K|u|$ and $|v(-d+ eu-fv)|\leq K|v|$, for some $K$.
Moreover, since the constant $C$ depend on the structure of the equation only, and \eqref{eq-main} switches between two systems \eqref{e6+} and \eqref{e6-}, $C$ can be chosen to be non-random.
Therefore, the proof is complete.
\end{proof}



In the next step, we show that $u(t,x)$ and $v(t,x)$ are greater than some positive constant, for all $x$, infinite often at the switching time $\tau_k$, which are the random times at which the switching process changes from one state to another.
\begin{prop}\label{prop-1}
	There exists $\hat \delta>0$ such that with probability 1, there are infinitely many $k=k(\omega)$ such that
	$$
\inf_E u(\tau_{2k+1},x)\geq \hat \delta,\;	\inf_E v(\tau_{2k+1},x)\geq \hat \delta.
	$$
\end{prop}

Let $$(u^*_+,v^*_+)=\Big(\frac {a(+)f(+)+c(+)d(+)}{b(+)f(+)+c(+)e(+)},\frac {a(+)e(+)-b(+)d(+)}{b(+)f(+)+e(+)c(+)}\Big),$$
$$(u^*_-,v^*_-)=\Big(\frac {a(-)f(-)+c(-)d(-)}{b(-)f(-)+c(-)e(-)},\frac {a(-)e(-)-b(-)d(-)}{b(-)f(-)+e(-)c(-)}\Big).$$
We divide the proof of this proposition into three cases, which are formalized as following lemmas.
\begin{lem}\label{lem-45}
	If both $(u^*_+,v^*_+)$ and $(u^*_-,v^*_-)$ are stable with respect to \eqref{e6+} and \eqref{e6-}, respectively, then conclusion in Proposition \ref{prop-1} holds.
\end{lem}	
	\begin{proof} 
	We first prove the fact that if $\inf_Eu_0(x),\inf_Ev_0(x)\geq \eps>0$ for some $\eps$ then there is $\eps_1=\eps_1(\eps)$ such that
	\begin{equation}\label{eq-327}
	\inf_Eu_{\pm}(t,x)\geq \eps_1,\;\inf_Ev_{\pm}(t,x)\geq\eps_1,\quad\forall t\geq 0.
	\end{equation}
	Indeed, by virtue of Lemma \ref{lem-com}, there is $T^*> 0$ such that $\inf_E u_{\pm}(t,x) > \frac{u^*_{\pm}}2$, $\inf_E v_{\pm}(t,x) > \frac{v^*_{\pm}}2$ for all $t > T^*$,
	provided $(u_{\pm}(0,x), v_{\pm}(0,x)) \in \C_{\eps,M}$. 
	Moreover, by comparison principle between $u_{\pm}(t,x), v_{\pm}(t,x)$ and solution to
	\begin{equation}\label{eq-ode-below}
		\begin{cases}
			U'(t)=U(t)(a_{\min}-b_{\max}M-c_{\max}M),\\
			V'(t)=V(t)(-d_{\max}-f_{\max}M),
		\end{cases}
	\end{equation}
	 we obtain that 
	$ u_{\pm}(t,x)\geq \eps e^{-KT^*},\; v_{\pm}(t,x)\geq \eps e^{-KT^*}\text{ for all } t\leq T^*,$
where $K=d_{\max}+(b_{\max}+c_{\max}+f_{\max})M$.
As a result, \eqref{eq-327} is proved.

	Now, by Lemma \ref{lem-inf}, there exists a sequence $s_n \uparrow\infty$ such that $\inf_E u(s_n,x)\geq \delta_1$ and $\inf_E v(s_n,x)\geq \delta_1$ for all $n$, for some $\delta_1>0$.
	Put $k_n := \min\{k: \tau_k>s_n\}$. 
	Therefore, it is seen that $\inf_Eu(\tau_{k_n},x)>\delta_2>0$ and $\inf_Ev(\tau_{k_n},x)>\delta_2>0$ for some $\delta_2$. Applying
	this again deduces that $\inf_Eu_{k_{n+1}}(x) >\delta_3>0$ and $\inf_Ev_{k_{n+1}}(x) >\delta_3>0$ for some $\delta_3$. 
	Obviously, the set $\{k_n: n \in\mathbb N\} \cup 
	\{k_n +1: n \in\mathbb N\}$ contains infinitely many odd numbers. The proof of the lemma is complete. 
	\end{proof}
	 \begin{lem}\label{lem-sus}
		If one system is stable and the other is unstable, then conclusion of Proposition \ref{prop-1} holds.
	\end{lem}
\begin{proof}
Assume that $(u^*_+,v^*_+)$ is stable point of \eqref{e6+} and  $(\frac{a(-)}{b(-)},0)$ is stable point of \eqref{e6-}.
	It is similar to Lemma \ref{lem-45} to prove that provided $\inf_Eu_0(x),\inf_Ev_0(x)\geq \eps>0$, 
	\begin{equation}\label{eq-lm36-1}
\inf_Eu_{+}(t,x)\geq\eps_1,\;		\inf_Ev_{+}(t,x)\geq\eps_1, \text{ for all } t\geq 0,
		\end{equation} 
for some $\eps_1=\eps_1(\eps)>0.$
	
We next prove the claim that: for any $\eps>0$, there is $\eps_1=\eps_1(\eps)$ such that 
\begin{equation}\label{eq-lm36-2}
\text{for any initial condition } v_0\in\C_{\eps_1,M}, v_-(t,x)\leq \eps\text{ for all }t\geq 0.
\end{equation}
Indeed, since $(u_-(t,x),v_-(t,x))$ tends to $(\frac{a(-)}{b(-)},0)$, using Lemma \ref{lem-com}, there is $T^*$ such that for any initial condition, $(u_0,v_0)\in \C_{0,M}\times C_{0,M}$,  $v_-(t,x)\leq \eps$ for all $t\geq T_*$.
Again, 	by comparison principle between $ v_{-}(t,x)$ and solution to
$		V'(t)=V(t)(e_{\max}M),$
we obtain that  $v_{-}(t,x)\leq e^{e_{\max}MT_*}\sup_{x}v_0(x)$ for all $t\leq T_*$. Choosing $\eps_1$ such that $e^{e_{\max}MT_*}\eps_1<\eps$,   $v_-(t,x)\leq \eps$ for all $t\geq 0$ provided $v_0\in\C_{\eps_1,M}$.

Now, by Lemma \ref{lem-inf}, there is a sequence $(s_n)\uparrow\infty$ such that $\inf_E u(s_n,x)\geq \delta_1>0$ and
$\inf_E v(s_n,x)\geq \delta_1>0$ for all $n$, for some $\delta_1>0$. 
If $\tau_{2k} \leq s_n \leq \tau_{2k+1}$ we have $\inf_Ev(\tau_{2k+1},x) >\eps_1>0$ for some $\eps_1>0$ due to \eqref{eq-lm36-1}. 
If $\tau_{2k+1} < s_n <\tau_{2k+2}$ then $\xi_{s_n} = -$. Let $\eps_2=\eps_2(\delta_1/2)$ as in \ref{eq-lm36-2}.
Then, we must have $\inf_Ev(\tau_{2k+1},x)\geq \eps_2$ because otherwise, $v(s_n,x)\leq \delta_1/2$ for all $x$, which is a contradiction.
	Let $\delta=\min\{\eps_1,\eps_2\}$, the proof is complete.
	The conclusion for $u_+(t,x)$ is proved similarly as done in Lemma \ref{lem-45}.
\end{proof}

\begin{lem}
	If both systems are unstable, then conclusion of Proposition \ref{prop-1} holds.
\end{lem}
\begin{proof}
	As in Lemma \ref{lem-sus}, we obtain that for any $\eps>0$, there is $\eps_1=\eps_1(\eps)$ such that for any initial condition $u_0,v_0\in\C_{\eps_1,M},$
	\begin{equation}\label{eq-lm37-2}
		 u_{\pm}(t,x)\leq \eps\text{ and }v_{\pm}(t,x)\leq \eps,\;\forall t\geq 0.
	\end{equation}
	Now, by Lemma \ref{lem-inf}, there is a sequence $(s_n)\uparrow\infty$ such that 	$\inf_E u(s_n,x)\geq \delta_1>0$ and
	$\inf_E v(s_n,x)\geq \delta_1>0$ for all $n$, for some $\delta_1>0$. 
There is $\tau_{2k+1}$ such that $\tau_{2k+1} < s_n$. Let $\eps_1=\eps_1(\delta_1/2)$ as in \ref{eq-lm37-2}.
	Then, we must have $\inf_Ev(\tau_{2k+1},x)\geq \eps_1$ because otherwise, $v(s_n,x)\leq \delta_1/2$ for all $x$, which is a contradiction. Similarly, $\inf_Eu(\tau_{2k+1},x)\geq \eps_1$.
	The proof is complete.
\end{proof}

We are now in position to describe the pathwise dynamic behavior of the solutions of system \eqref{eq-main}. We consider the following systems of ODE:
\begin{equation}\label{e6+ode}
	\begin{cases}
		\dfrac{dU_+(t)}{d t}=U_+(t)(a(+)- b(+)U_+(t)-c(+)V_+(t)),\\[1ex]
		\dfrac{dV_+(t)}{dt}=V_+(t)(-d(+)+ e(+)U_+(t)-f(+)V_+(t)),
\end{cases}
\end{equation}
and
\begin{equation}\label{e6-ode}
	\begin{cases}
		\dfrac{dU_-(t)}{dt}=U_-(t)(a(-)- b(-)U_-(t)-c(-)V_-(t)),\\[1ex]
		\dfrac{dV_-(t)}{dt}=V_-(t)(-d(-)+ e(-)U_-(t)-f(-)V_-(t)).
\end{cases}
\end{equation}
For $(U_0,V_0)\in\mathbb R^2$, $U_0\geq 0, V_0\geq 0$, we denote 
$\pi^+_t (U_0, V_0) = (U_+(t), V_+(t))$
 (resp. $\pi^-_t (u, v) =(U_-(t), V_-(t))$ the solution of system \eqref{e6+ode} (resp. \eqref{e6-ode}) with initial conditions $(U_+(0),V_+(0))=(U_0,V_0)$ (resp. $(U_-(0),V_-(0))=(U_0,V_0)$).
 Let 
\begin{equation*}\label{eS}
S=\left\{(u, v)=\pi_{t_n}^{\varrho(n)}\cdots\pi_{t_1}^{\varrho(1)}\pi_{t_0}^{+}(u^*_+,v^*_+):0\leq t_0,t_1,t_2,\ldots,t_n; \; n\in \NN\right\},
\end{equation*}
where $\varrho(k)=+$ if $k$ is even, and $\varrho(k)=-$ if $k$ is odd. 

It is noted that $S$ is a subset of $\mathbb R^n$.
By abuse of notation, we still denote by $S$ the subset of $C(\bar E)$ that contains all constant functions taking values in $S$.
In the following theorem and its proof, $S$ will be understood either as subset of $\mathbb R^n$ or subset of $C(\bar E)$, and it should be clear from the context.
We make the same convention for $\pi_{t_0}^+(u^*_+,v^*_+),\dots$.
\begin{thm}\label{thm2}
With probability 1, the closure $\bar {S}$ of $S$ is a subset of the $\omega$-limit set $\Omega(u_0,v_0,\omega)$.
\end{thm}

\begin{proof}
We construct a sequence of stopping times
\begin{eqnarray*}
\eta_1&=&\inf\{2k+1: (u(\tau_{2k+1},x), v(\tau_{2k+1},x))\in 
\C_{\delta,M}\},\\
\eta_2&=&\inf\{2k+1>\eta_1: (u(\tau_{2k+1},x), v(\tau_{2k+1},x))\in \C_{\delta,M}\},\\
\cdots&&\\
\eta_n&=&\inf\{2k+1>\eta_{n-1}: (u(\tau_{2k+1},x), v(\tau_{2k+1},x))\in \C_{\delta,M}\},\ldots
\end{eqnarray*}
It is easy to see that  $\{\eta_k=n\}\in\mathscr{F}_0^n$ for any $k, n$. Thus, the event $\{\eta_k=n\}$ is independent of $\mathscr{F}_n^{\infty}$ if $\xi_0$ is given. By Proposition \ref{prop-1}, $\eta_n<\infty$ a.s. for all $n$.

 For any $k\in\mathcal N$ and $s_1 > 0$, $s_3>s_2 > 0$, $s_4>0$,
let 
$$B_k:=\{\omega: \sigma_{\eta_k+1}<s_1, s_2<\sigma_{\eta_k+2}<s_3, \sigma_{\eta_k+3}<s_4
\}.$$
Using arguments similar to the proof of \cite[Theorem 2.2]{du2011dynamics}, we obtain that
${\mathbb P}\biggl(\bigcap_{k=1}^{\infty}\bigcup_{i=k}^{\infty}B_i\biggl)=1.$
As a result,
\begin{equation}\label{eq-i.o}
	{\PP}\{\omega: \sigma_{\eta_n+1}<s_1, s_2<\sigma_{\eta_n+2}<s_3, \sigma_{\eta_n+3}<s_4 \mbox{ i.o. of}\;n\}=1.
\end{equation}

By comparing $(u(t,x),v(t,x))$ with equation \eqref{eq-ode-below} as in proof of Lemma \ref{lem-45}, we obtain that: given $\inf_Eu(\tau_{\eta_k},x)\geq \eps$ and $\inf_Ev(\tau_{\eta_k},x)\geq \eps$, then $\inf_Eu(t+\tau_{\eta_k},x)\geq \eps e^{-Kt}$ and $\inf_Ev(t + \tau_{\eta_k},x) \geq \eps e^{-Kt}$ for all $t > 0$.

Let $U_\eps$ is any $\eps$-neighborhood (in $\mathbb R^2$) of the $(u^*
_+, v^*_+)$. 
By Lemma \ref{lem-com}, there is $T > 0$ such that
$(u_+(t, x), y_+(t, x)) \in U_{\eps}\;\forall x$ for any $t > T$ provided that the initial condition $(u_0, v_0) \in \C_{\delta,M}$. Therefore, $(u(\tau_{\eta_k+2},x),v(\tau_{\eta_k+2},x)) \in U_\eps\;\forall x$ provided
$\sigma_{\eta_k+1} < s$, $\sigma_{\eta_k+2} > T.$ As a result, $(u(\tau_{2k+3},x), v(\tau_{2k+3},x)) \in U_\eps\;\forall x$ for infinite many $k$. This means that the constant functions $(x^*_+, y^*_+)\in 
\Omega(u_0, v_0,\omega)$ a.s.

The next step is to prove that $\{\pi^-_t(u^*_+, v^*_+):t\geq 0\}\subset \Omega(u_0,v_0,\omega)$ a.s. 
Let $t_1$ be an arbitrary positive number.
 In view of the continuity of the solution in the initial values, for any neighborhood $U_{\eps_1}(\pi^-_{t_1}(u^*_+, v^*_+))$ in $\mathbb R^2$ of $\pi^-_{t_1}(u^*_+, v^*_+)$, there are $0<t_2<t_1<t_3$ and $\delta_2>0$ such that if $(u_0, v_0)$ is such that $(u_0(x),v_0(x))$ belongs to the neighbor $U_{\eps_2}(u^*_+,v^*_+)$ of $(u^*_+,v^*_+)$ then $(u_-(t,x),v_-(t,x))\in U_{\eps_1}(\pi^-_{t_1}(u^*_+, v^*_+))\;\forall x$  for any $t_2<t<t_3.$
Put
\begin{eqnarray*}
\zeta_1&=&\inf\{2k+1: (u(\tau_{2k+1},x), v(\tau_{2k+1},x))\in U_{\eps_2}\},\\
\zeta_2&=&\inf\{2k+1>\zeta_1: (u(\tau_{2k+1},x), v(\tau_{2k+1},x))\in U_{\eps_2}\},\\
\cdots&&\\
\zeta_n&=&\inf\{2k+1>\zeta_{n-1}: (u(\tau_{2k+1},x), v(\tau_{2k+1},x))\in U_{\eps_2}\}\ldots
\end{eqnarray*}
From the previous part of this proof, it follows that $\zeta_k<\infty$  and $\lim\limits_{k\to\infty}\zeta_k=\infty$ a.s.. Since $\{\zeta_k=n\}\in\mathscr{F}_0^n$, $\{\zeta_k\}$ is independent of $\mathscr{F}_n^{\infty}.$ 
By the same argument as \cite[Theorem 2.2]{du2011dynamics}, we obtain $${\mathbb P}\{\omega: \sigma_{\zeta_n+1}\in(t_2, t_3) \mbox{ i.o. of \;}n\}=1.$$ This relation means that  
$(u(\tau_{\zeta_k+1},x), v(\tau_{\zeta_k+1},x))\in U_{\eps_1}(\pi^-_{t_1}(u^*_+, v^*_+))\;\forall x$ for many infinite  $k\in \NN$. This means that $\pi^-_{t_1}(u^*_+, v^*_+)\in \Omega(u_0,v_0,\omega)$ a.s. 

Similarly, for any $t>0$, the orbit  $\{\pi_s^+\pi_t^-(u^*_+, v^*_+): s>0\}\subset \Omega(u_0,v_0,\omega)$. By induction, we conclude that $S$ is a subset of $\Omega(u_0, v_0).$ It follows from the closedness of  $\Omega(u_0,v_0,\omega)$ that $\Bar{S}\subset \Omega(u_0,v_0,\omega)$ a.s.
\end{proof}

\section{Numerical Simulations}\label{sec:num}
In this section, we present some numerical examples to illustrate our theoretical results.
These examples will also demonstrate interesting effects of discrete events process $\xi(t)$.

\begin{exam}\label{exam5.2}
	We examine system \eqref{eq-main} in $E=[0,1]$.
	The other parameters are
	$q(+)=q(-)=5$,
	$\alpha_1(+)=\alpha_1(-)=\alpha_2(+)=\alpha_2(-)=1$,
	$a(+)=1$,
	$b(+)=1$,
	$c(+)=1$,
	$d(+)=2$,
	$e(+)=1$,
	$f(+)=1$,
	$a(-)=3$,
	$b(-)=1$,
	$c(-)=1$,
	$d(-)=7$,
	$e(-)=2$,
	$f(-)=1$; the initial conditions are $u_0(x)=2\cos (\pi x)+2$, and $v_0(x)=2\sin^2(\pi x)$.

Our theoretical results allow us to characterize the long-term behavior of the system through its $\omega$-limit set, which is determined by the quantity $\lambda$ defined in \eqref{e8}. In this example, direct computation yields $\lambda = -3$. Therefore, by Theorems \ref{thm-ext} and \ref{thm-omega-ext}, we conclude that $v(t,x) \to 0$ uniformly, and the $\omega$-limit set of $u(t,x)$ is $\mathcal{K}^*$, the space of constant functions taking values in the interval $\left[\frac{a(+)}{b(+)}, \frac{a(-)}{b(-)}\right]=[1,3]$.
This behavior is illustrated in Figures \ref{f5}. Starting from the initial condition $u_0(x) = 2\cos(\pi x) + 2$ and $v_0(x) = 2\sin^2(\pi x)$ at $t = 0$, the solution $(u(t,x), v(t,x))$ approaches its $\omega$-limit set $\mathcal{K}^*$ as $t$ increases. The simulations show that for $t \geq 1$, $u(t,x)$ becomes nearly spatially constant with values lying in $[1,3]$, while $v(t,x)$ becomes uniformly close to zero. The influence of environmental switching is also evident in the figures: although the limiting functions are constant in space, their values fluctuate within the interval $[1,3]$ due to the stochastic switching dynamics.

	\begin{figure}[H]
		\subfigure{\includegraphics[height=0.32\textheight,width=0.52\textwidth]{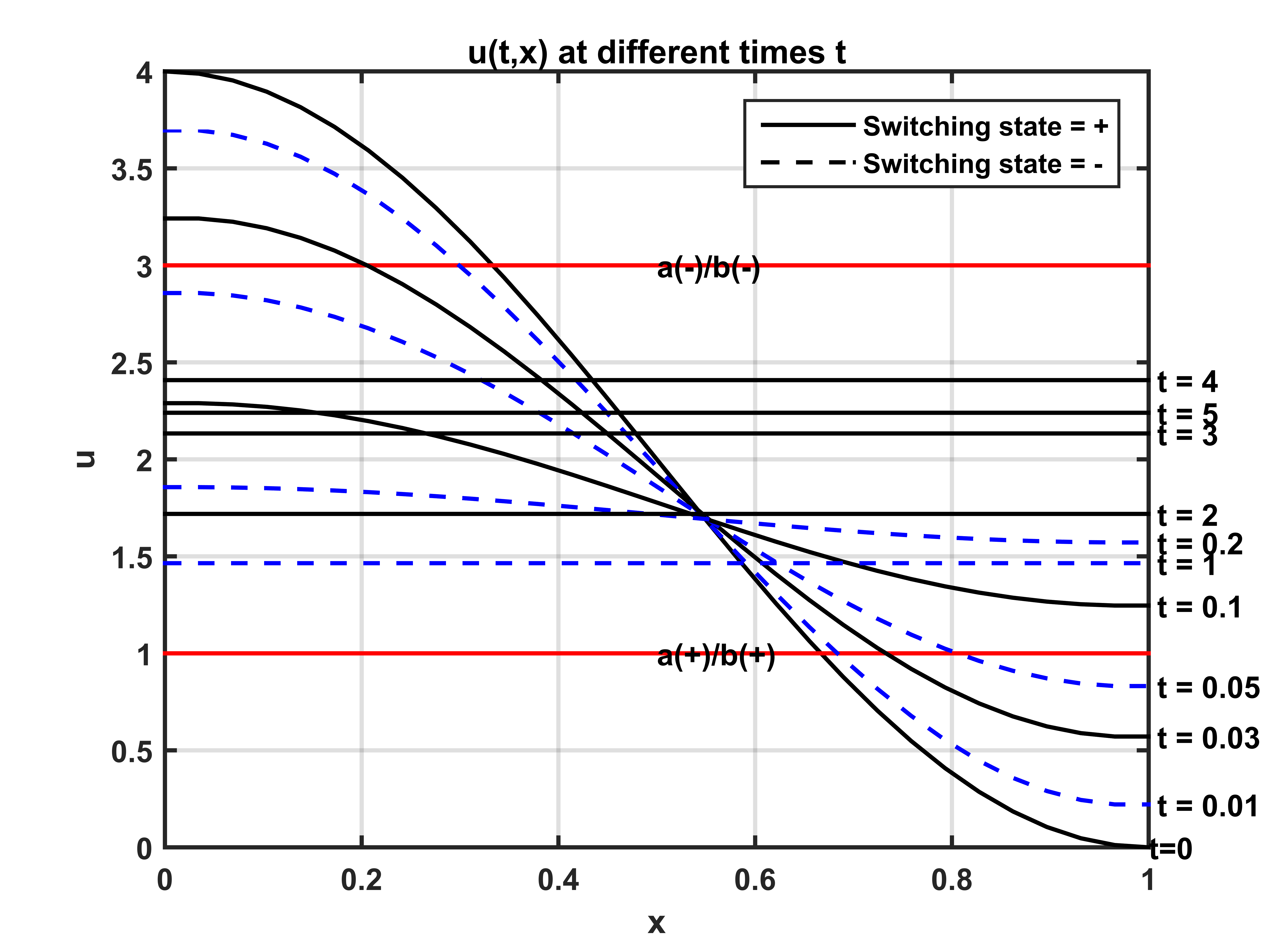}}
		\subfigure{\includegraphics[height=0.32\textheight,width=0.52\textwidth]{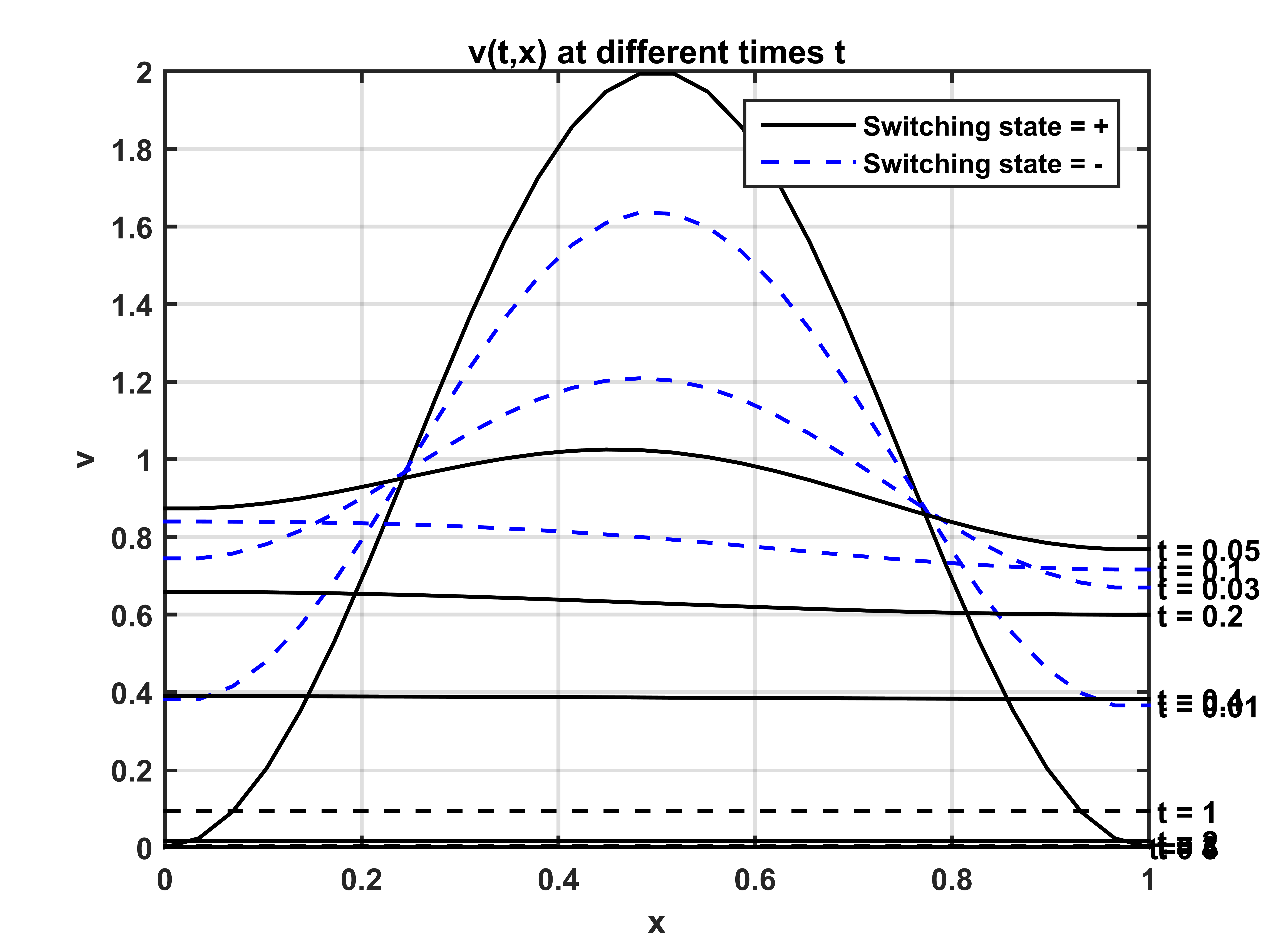}}
        \subfigure{\includegraphics[height=0.32\textheight,width=0.52\textwidth]{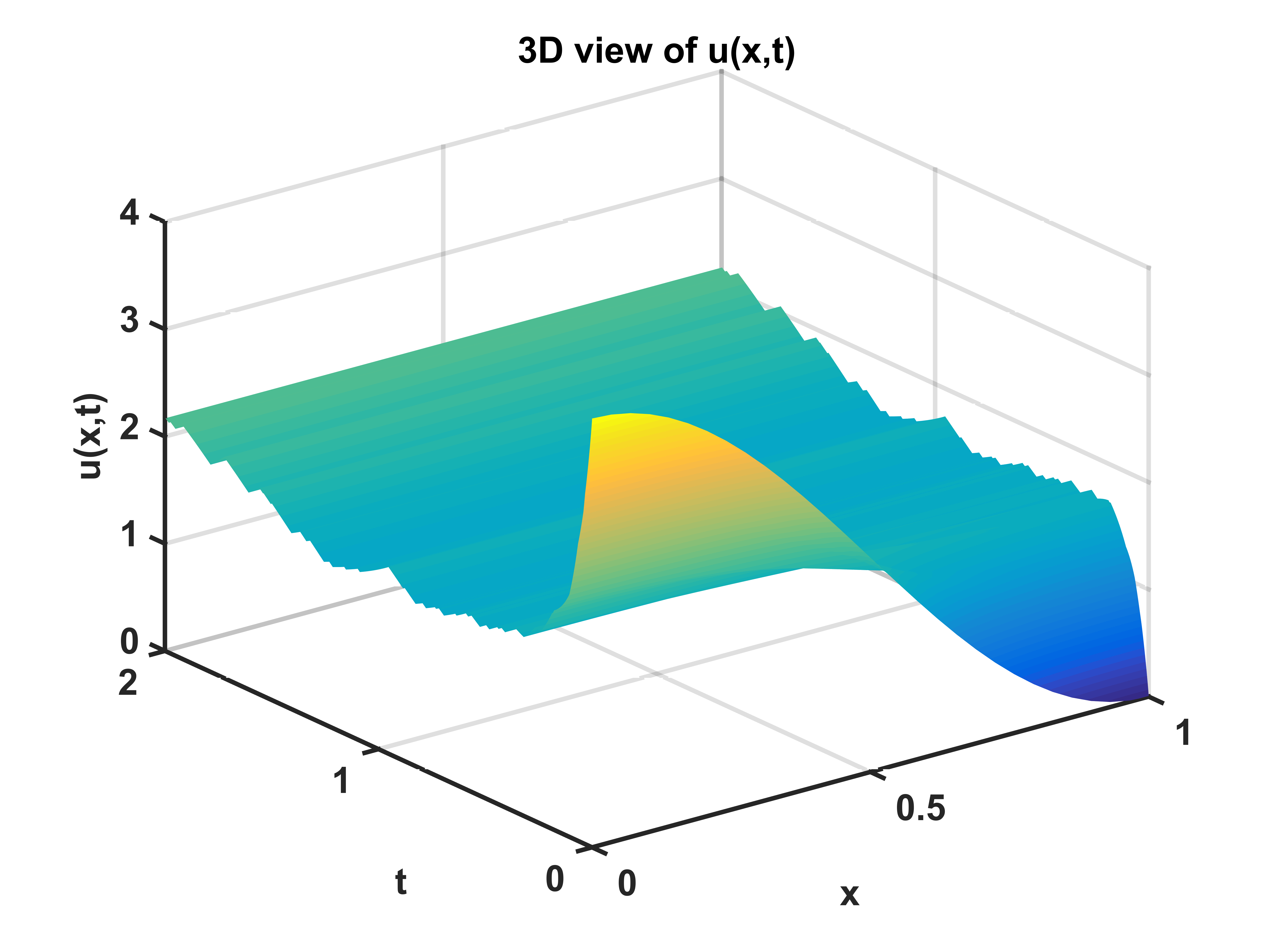}}
	\subfigure{\includegraphics[height=0.32\textheight,width=0.52\textwidth]{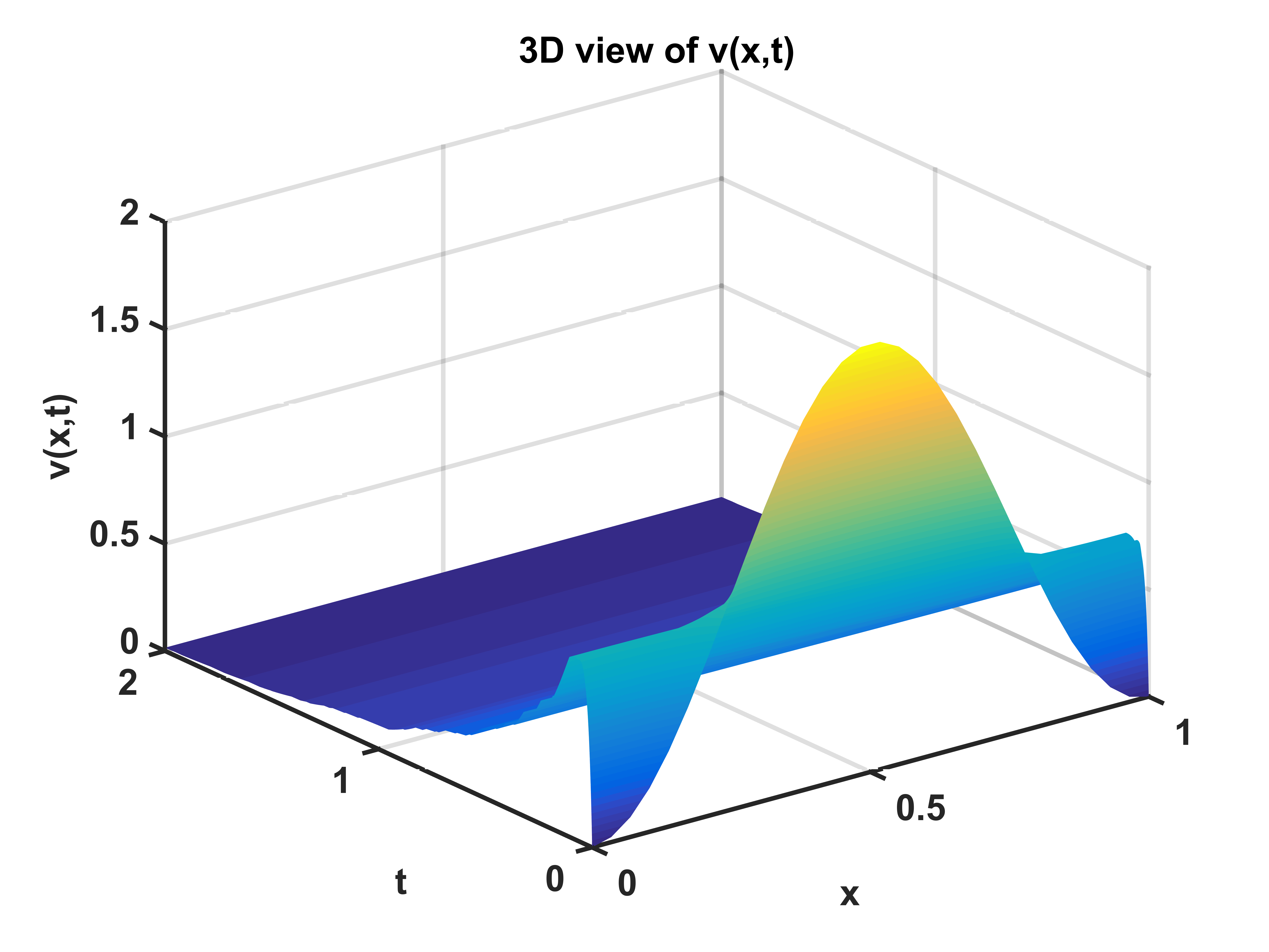}}
		\caption{2D and 3D simulations for $u(t,x)$ and $v(t,x)$ in Example \ref{exam5.2}.}\label{f5}
	\end{figure}
\end{exam}

\begin{exam}\label{exam5.3}
	We examine system \eqref{eq-main} in $E=[0,1]$.
	The other parameters are
	$q(+)=q(-)=5$,
	$\alpha_1(+)=\alpha_1(-)=\alpha_2(-)=\alpha_2(-)=1$,
	$a(+)=5$,
	$b(+)=1$,
	$c(+)=1$,
	$d(+)=0.5$,
	$e(+)=2$,
	$f(+)=1$,
	$a(-)=15$,
	$b(-)=1$,
	$c(-)=1$,
	$d(-)=1$,
	$e(-)=2$,
	$f(-)=3$; the initial conditions are $u_0(x)=5\cos (\pi x)+2$, and $v_0(x)=5\sin^2(\pi x)$.

	{\color{blue}
	In this example, 
	$\lambda=38.5$.
	Theorems \ref{thm-per} and \ref{thm2} show that \eqref{eq-main} is persistent and the solution approaches $\bar S$ (given in \eqref{eS}) as $t$ becomes larger and larger. 
	These results are shown in Figure \ref{f6}.
	Precisely, starting from the initial condition $(5\cos (\pi x)+2,5\sin^2(\pi x))$, for $t \geq 2$, $u(t,x)$ and $v(t,x)$ becomes (almost) constant with values fluctuate within the intervals, which are highlighted in red in the 2D figure.
	}
	\begin{figure}[H]
		\subfigure{\includegraphics[height=0.32\textheight,width=0.52\textwidth]{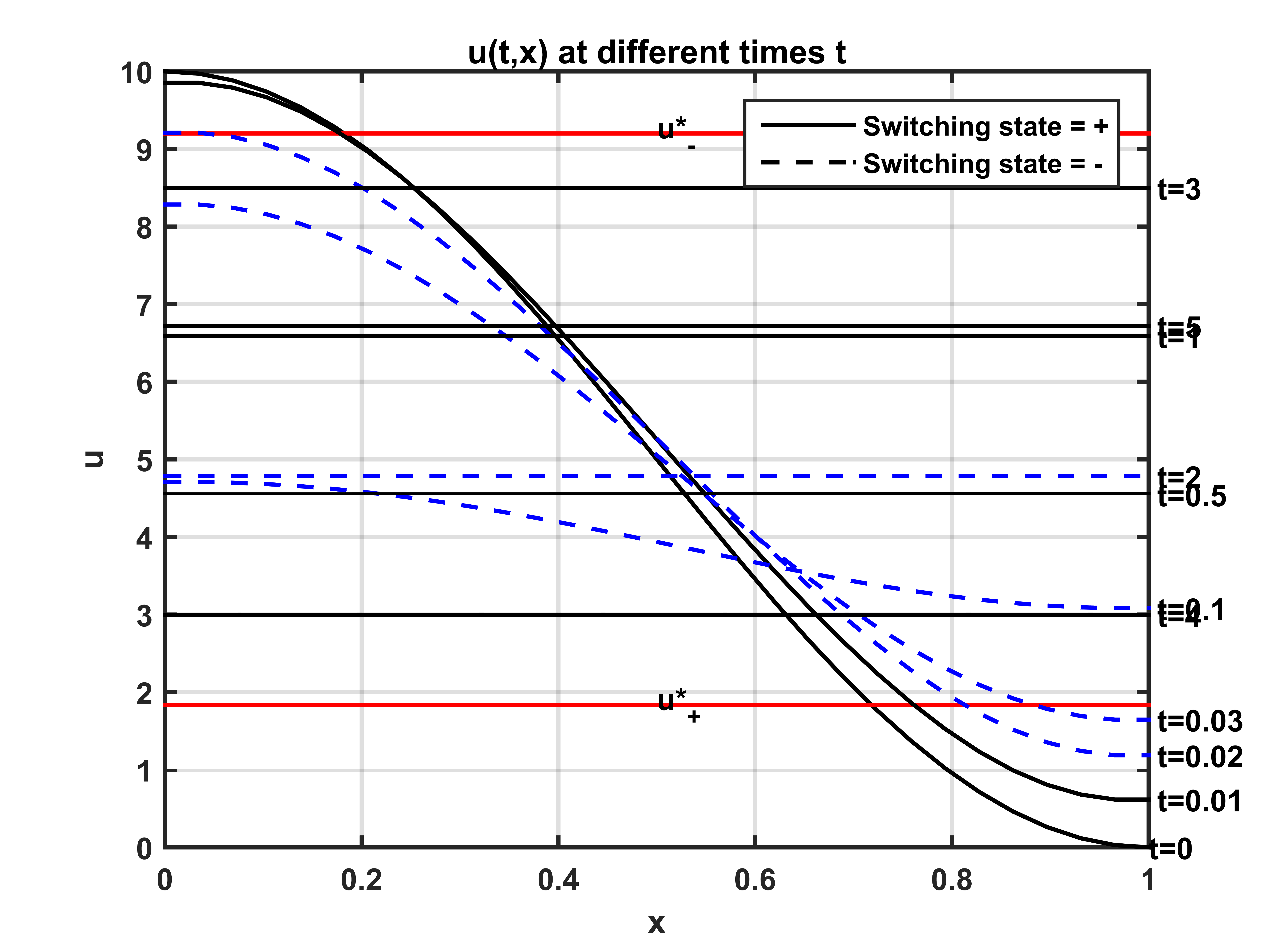}}
		\subfigure{\includegraphics[height=0.32\textheight,width=0.52\textwidth]{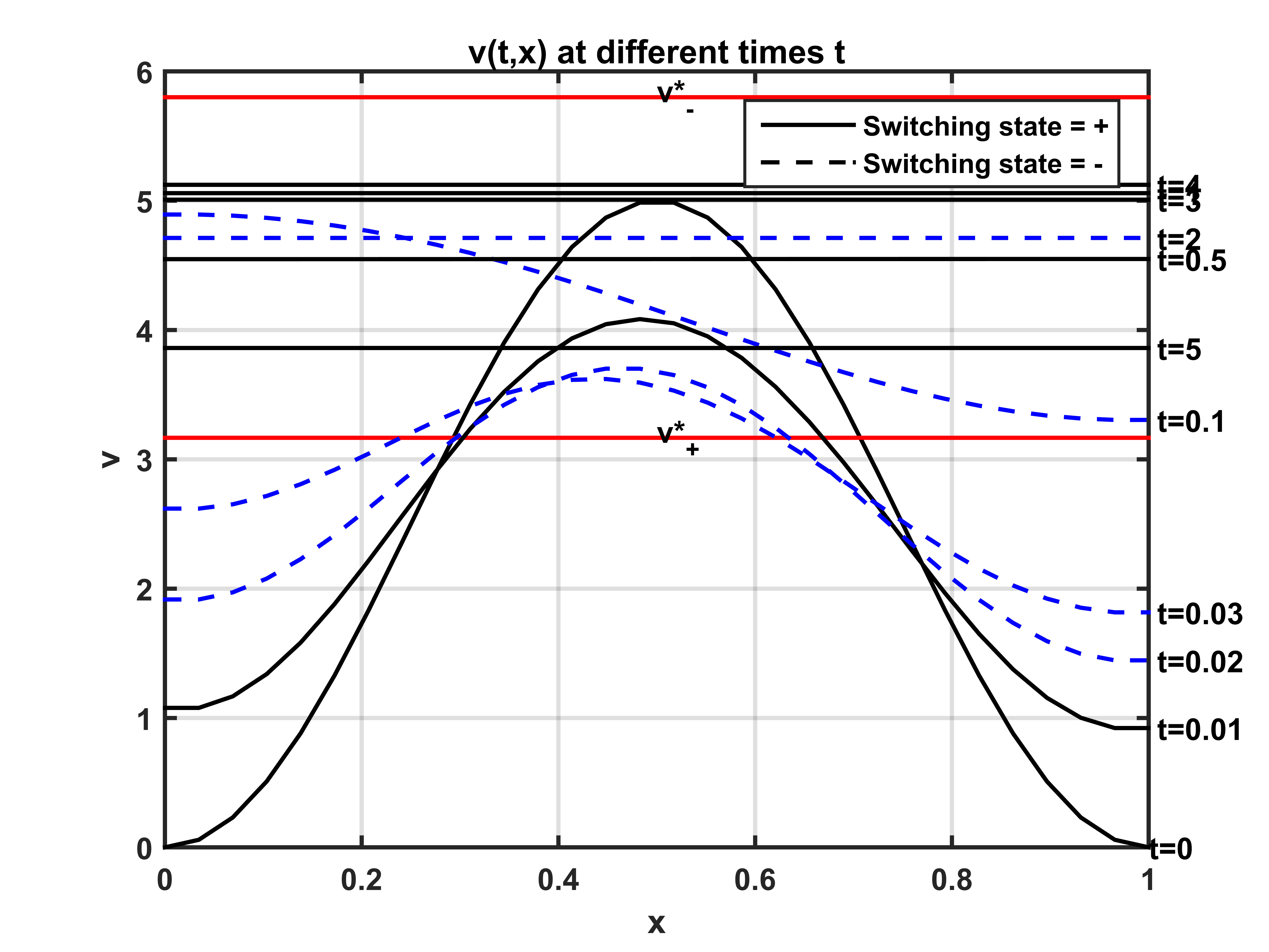}}
        \subfigure{\includegraphics[height=0.32\textheight,width=0.52\textwidth]{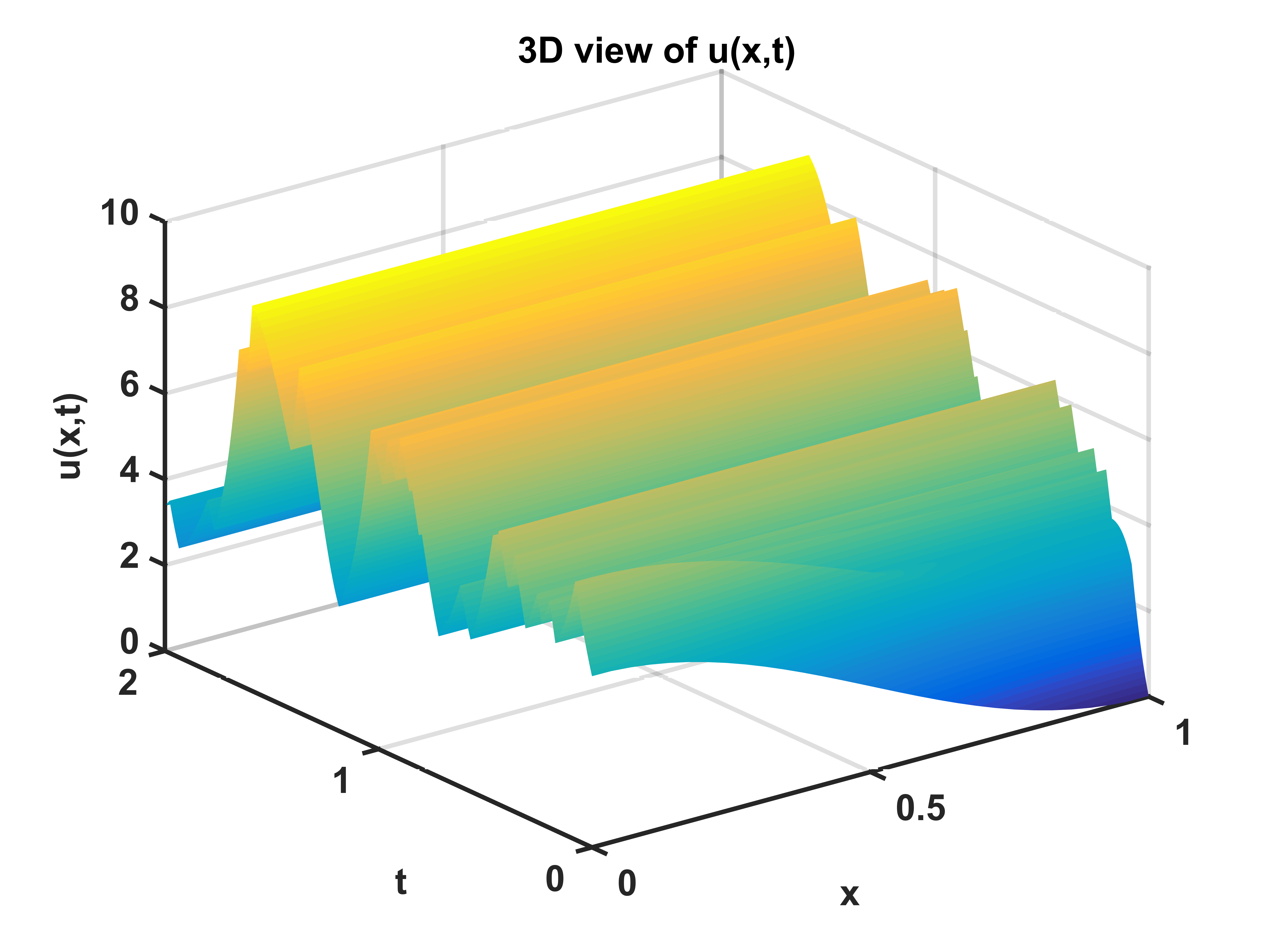}}
		\subfigure{\includegraphics[height=0.32\textheight,width=0.52\textwidth]{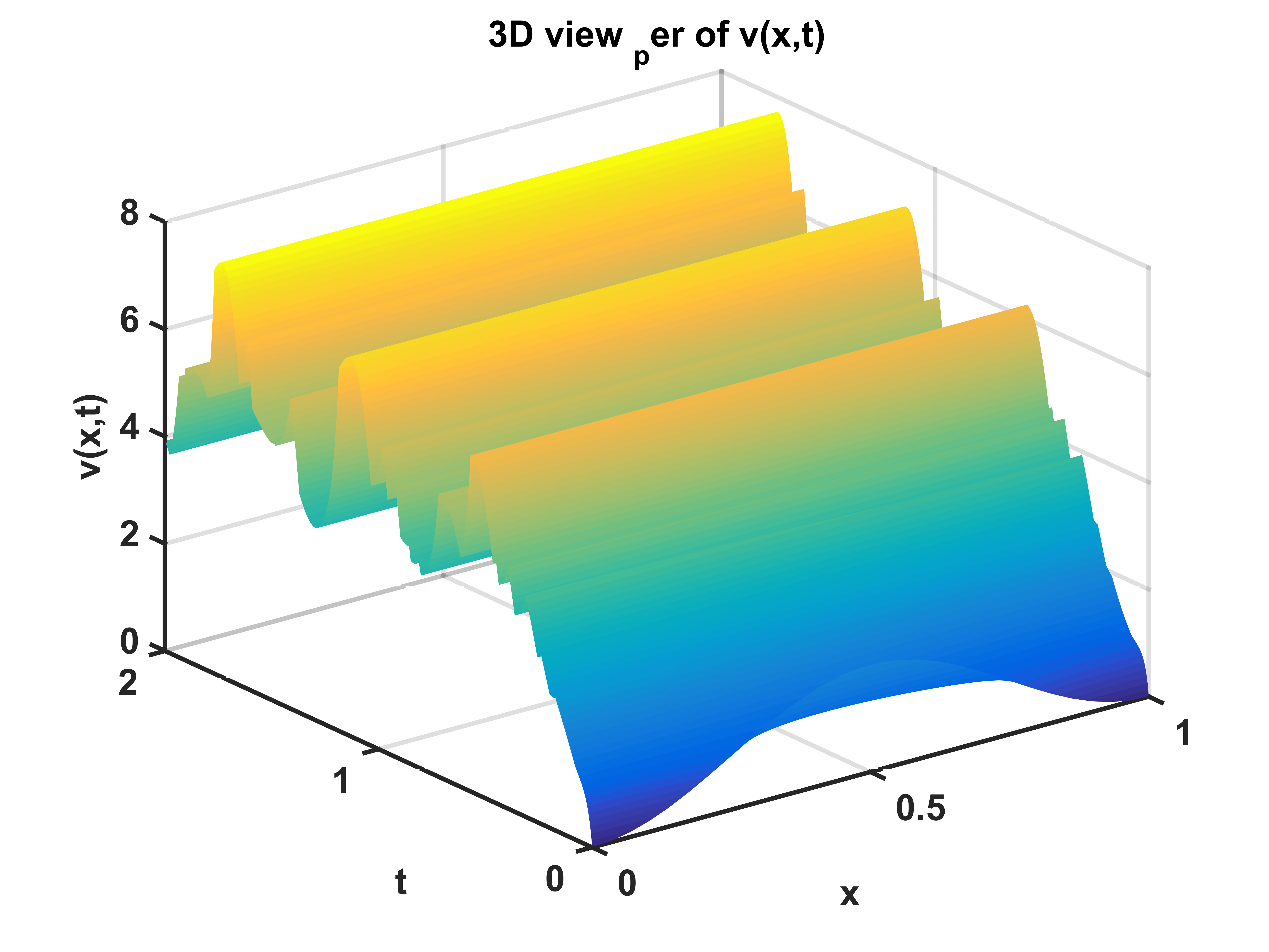}}
		\caption{2D and 3D simulations for $u(t,x)$ and $v(t,x)$ in Example \ref{exam5.3}. 
			}\label{f6}
	\end{figure}
\end{exam}

\section{Conclusions}\label{sec:con}
	In this paper, we introduced and analyzed a novel class of spatially heterogeneous predator-prey models governed by reaction-diffusion partial differential equations under random environmental switching. 
	In contrast to the existing studies, which either consider stochastic ODE model (that does not capture space evolution) or PDE model (that does not take randomness into the consideration),
	the incorporation of Markovian switching into the classical spatial predator-prey framework allows us to better capture the influence of abrupt and discrete environmental changes, such as seasonal shifts or extreme events on species dynamics. This hybrid modeling approach bridges deterministic reaction-diffusion systems and stochastic systems driven by discrete randomness, offering a more realistic and flexible description of ecological processes.
	
	Our contributions are twofold. First, we examined the hybrid stochastic PDE system to characterize the long-time behavior of solutions, identifying sharp threshold conditions that distinguish between extinction and persistence regimes for both prey and predator populations. These thresholds are determined by an appropriately defined Lyapunov exponent from dynamical system theory. Second, we characterized the pathwise $\omega$-limit set of the system and provided rigorous classifications of its attractors, offering a detailed understanding of the stochastic trajectories’ limiting behavior.
	
	In addition, we developed new analytical tools that go beyond traditional Lyapunov functionals, allowing for a deeper examination of pathwise dynamics under both spatial and temporal randomness. Our findings were further supported by numerical simulations, which demonstrate the rich and sometimes counterintuitive behaviors induced by random switching.
	Overall, our study not only advances the theoretical understanding of stochastic hybrid reaction-diffusion systems in ecology but also opens new directions for applying such models to broader classes of biological and physical phenomena where spatial structure and environmental variability play a crucial role.

Future work could extend these results to more complex interaction networks, nonlocal diffusion, or systems with delay and memory effects.

\bibliographystyle{plain}
\bibliography{PDE-prey-predator}
\end{document}